\documentclass[11pt]{article}
\usepackage{graphicx}
\usepackage{amsmath,amsfonts,amsthm}
\usepackage{amssymb,latexsym}
\usepackage{fixmath}
\usepackage{mathrsfs,amsbsy}
\usepackage{dsfont}
\usepackage{enumerate}

\DeclareMathOperator*{\argmin}{argmin}

\DeclareMathOperator{\diag}{diag}

\DeclareMathOperator{\rank}{rank}

\DeclareMathOperator{\HH}{H}
\DeclareMathOperator{\T}{T}

\def\ba{\pmb{a}}
\def\bb{\pmb{b}}

\def\be{\pmb{e}}
\def\Bf{\pmb{f}}

\def\bh{\pmb{h}}

\def\bk{\pmb{k}}

\def\bn{\pmb{n}}

\def\bp{\pmb{p}}

\def\br{\pmb{r}}
\def\bs{\pmb{s}}
\def\bt{\pmb{t}}

\def\bv{\pmb{v}}
\def\bw{\pmb{w}}
\def\bx{\pmb{x}}
\def\by{\pmb{y}}
\def\bz{\pmb{z}}

\newtheorem{proposition}{Proposition}[section]
\newtheorem{theorem}{Theorem}[section]
\newtheorem{lemma}{Lemma}[section]

\theoremstyle{definition}

\newtheorem{remark}{Remark}[section]
\newtheorem{example}{Example}[section]

\numberwithin{equation}{section}
\numberwithin{figure}{section}
\numberwithin{table}{section}

\allowdisplaybreaks

\usepackage{kbordermatrix}

\usepackage{amsfonts,amsthm}
\usepackage{amsmath,amssymb,enumerate,color}
\usepackage{algorithm,algorithmic}
\usepackage{epsfig,graphicx,psfrag,mathrsfs,subfigure}
\usepackage{eucal}
\usepackage{yhmath}
\usepackage{hyperref}
\hypersetup{colorlinks}
\RequirePackage{multirow}
 
\usepackage{textcomp}
\usepackage[]{fontenc}
\usepackage{extarrows}

\usepackage{rotating}
\usepackage{lscape}
\usepackage{bm}
\usepackage{xspace}
\usepackage{tikz}
\usepackage{geometry}
\def\ba{\pmb{a}}
\def\bb{\pmb{b}}

\def\be{\pmb{e}}
\def\Bf{\pmb{f}}

\def\bh{\pmb{h}}

\def\bk{\pmb{k}}

\def\bn{\pmb{n}}

\def\bp{\pmb{p}}

\def\br{\pmb{r}}
\def\bs{\pmb{s}}
\def\bt{\pmb{t}}

\def\bv{\pmb{v}}
\def\bw{\pmb{w}}
\def\bx{\pmb{x}}
\def\by{\pmb{y}}
\def\bz{\pmb{z}}

\def\tu{\mathfrak{u}}

\def\bbC{\mathbb{C}}

\def\bbP{\mathbb{P}}
\def\bbR{\mathbb{R}}

\renewcommand{\algorithmicrequire}{\textbf{Input:}}
\renewcommand{\algorithmicensure}{\textbf{Output:}}

\numberwithin{equation}{section}
\numberwithin{figure}{section}
\numberwithin{table}{section}

\graphicspath{{figures/}{figure/}{pictures/}%
	{picture/}{pic/}{pics/}{image/}{images/}}

\title{The $L_q$-weighted dual programming of the linear Chebyshev approximation and an interior-point method}
\author{Linyi Yang\thanks{School of Mathematical Sciences, Soochow University, Suzhou 215006, Jiangsu, China. Email: {\tt lyyang161@stu.suda.edu.cn}.}\and Lei-Hong Zhang\thanks{Corresponding author. School of Mathematical Sciences, Soochow University, Suzhou 215006, Jiangsu, China. This work was
 supported in part by the National Natural Science Foundation of China
        NSFC-12071332.
        Email: {\tt longzlh@suda.edu.cn}.}\and Ya-Nan Zhang\thanks{School of Mathematical Sciences, Soochow University, Suzhou 215006, Jiangsu, China.  Email: {\tt ynzhang@suda.edu.cn}.}}

\date{ }

\begin{document}

\maketitle

\begin{abstract}
Given  samples of a real or complex-valued function on a set of distinct nodes, the traditional linear Chebyshev approximation is to compute the best minimax approximation on a prescribed linear functional space. Lawson's iteration is a classical and well-known method for that task. However,  Lawson's iteration converges linearly and in many cases, the convergence is very slow. In this paper, by the duality theory of linear programming, we first provide  an elementary and self-contained proof for the well-known Alternation Theorem in the real case. Also, relying upon the Lagrange duality, we further establish an $L_q$-weighted  dual programming for the linear Chebyshev approximation. In this framework, we revisit the convergence of Lawson's iteration, and moreover, propose a Newton type iteration, the interior-point method,  to solve the $L_2$-weighted  dual programming.  Numerical experiments are reported to demonstrate its fast convergence and its capability in finding  the reference points that  characterize the unique minimax approximation.
\end{abstract}


\medskip
{\small
\medskip

}



\section{Introduction}\label{sec_intro}

Let $\Omega$ be a compact subset of the complex plane $\bbC$ and $f(x)$ be  a complex-valued function on $\Omega$.
For a given set of  nodes $\{(x_j,f_j)\}_{j=1}^m$ sampled from  $f_j=f(x_j)\in \bbC$ ($x_j\in \Omega$) over a distinct set, we consider the best linear Chebyshev  approximation of finding a function $p^*\in \bbP_n={\rm span}(\psi_1,\dots,\psi_n)$ defined by
\begin{equation}\label{eq:bestf}
p^*=\argmin_{p\in\bbP_n}\|\Bf-p(\bx)\|_{\infty},
\end{equation}
where $\Bf=[f_1,\dots,f_m]^{\T}\in\bbC^m ~(n+1\le m)$,  $\bx=[x_1,\dots,x_m]^{\T}\in\bbC^m$,
\begin{equation}\nonumber
\|\Bf-p(\bx)\|_{\infty}=\max_{1\le j\le m}|f_j-p(x_j)|
\end{equation}
and $\{\psi_j\}_{j=1}^n$ are complex-valued functions defined on $\Omega$ satisfying the Haar condition\footnote{For complex-valued continuous functions $\psi_i$ defined on $\Omega$, we say the set $\{\psi_i\}_{i=1}^n$  satisfies the Haar condition \cite[Chapter  3.4]{chen:1982}   if the matrix $[\psi_{i}(x_j)]\in \bbC^{n\times n}$ is invertible for arbitrary $n$ distinct points $\{x_j\}_{j=1}^n$ in $\Omega$.}  (aks the Chebyshev system or the Chebyshev set) \cite[Chapter  3.4]{chen:1982} and \cite{rice:1969}.
Under these assumptions, it is known that the solution $p^*$ of \eqref{eq:bestf} is unique (see e.g., \cite[Chapter 3.4]{chen:1982}, \cite[Chapter 4]{lore:1966}, \cite{pous:1911} or \cite[Problem 2]{rish:1961}). In particular, when $n=m$, then the solution $p^*$ is the interpolation function satisfying $p^*(x_j)=f_j$ and thus, we only concern with the case $m\ge n+1$ and $\|\Bf-p^*(\bx)\|_{\infty}>0$ in this paper.

A traditional method for solving \eqref{eq:bestf} is Lawson's iteration \cite{laws:1961}, in which the best approximation $p^*$ is achieved as the limit of a sequential weighted least-squares problems.  The basic Lawson's iteration can be summarized in Algorithm \ref{alg:lawson}.
\begin{algorithm}[t]
\caption{Lawson's iteration for \eqref{eq:bestf}} \label{alg:lawson}
\begin{algorithmic}[1]
\renewcommand{\algorithmicrequire}{\textbf{Input:}}
\renewcommand{\algorithmicensure}{\textbf{Output:}}
\REQUIRE $\{(x_j,f_j)\}_{j=1}^m$ and $0\le n+1\le m$ and a function basis $\{\psi_i\}_{i=1}^n$ for $\bbP_n$
\ENSURE  the (approximation) solution $p^*\in \bbP_n$ of \eqref{eq:bestf} 
        \smallskip

\STATE  Let $k=0$, $d^{(0)}=0$ and select $0<\bw^{(0)}=[w_1^{(0)},\dots,w_m^{(0)}]^{\T}$ and $\be^{\T}\bw^{(0)}=1$;
\STATE  Solve $$p^{(k)}=\arg\min_{p\in \bbP_n}\sum_{j=1}^m w^{(k)}_j|f_j-p(x_j)|^2$$
and compute $d^{(k)}=\sqrt{\sum_{j=1}^m w_j^{(k)}\left|f_j-p^{(k)}(x_j)\right|^2}$;
\STATE  Update $w_j^{(k+1)}=\frac{w_j^{(k)}\left|f_j-p^{(k)}(x_j)\right|}{\sum_{i=1}^m w_i^{(k)}\left|f_i-p^{(k)}(x_i)\right|}$ for $j=1,2,\dots,m$;
\STATE Stop if $\frac{|d^{(k)}-d^{(k+1)}|}{d^{(k)}}\le \epsilon$; otherwise $k=k+1$ and goto step 2.
\end{algorithmic}
\end{algorithm}

One  of principles behind Lawson's iteration is an elegant necessary and sufficient characterization (see e.g., \cite{rish:1961,will:1972}). 
In particular, for the real case (i.e., $\Omega\subseteq\bbR$, $f(x)$ and $\{\psi_j\}_{j=1}^n$ are real-valued functions), this characterization is known as the  Alternation Theorem (aks The Chebyshev Equioscillation Theorem) (see e.g., \cite[Chapter 3.4]{chen:1982} and also Theorem \ref{thm:Alternation}). Specifically, for a real valued $f$ defined in an  interval $\Omega=[a,b]$,   $p^*\in \bbP_n$ is the solution of \eqref{eq:bestf} if and only if there are at least $n+1$ reference points (we call these points as equioscillation points) $x_{\beta_1}<x_{\beta_2}<\dots<x_{\beta_{n+1}}$ in $X$ so that $p^*(x_{\beta_j})-f_{{\beta_j}}=(-1)^{j}\|p^*(\bx)-\Bf\|_{\infty}$ for $1\le j\le n+1$ or, $p^*(x_{\beta_j})-f_{{\beta_j}}=(-1)^{j+1}\|p^*(\bx)-\Bf\|_{\infty}$ for $1\le j\le n+1$. Basically, Lawson's iteration attempts to determine these equioscillation points or a subset of nodes that contains equioscillation points by boosting the associated weights and suppressing, simultaneously,  other weights corresponding to non-equioscillation points in an iterative scheme. 
The convergence of Lawson's iteration is discussed in \cite{laws:1961}, and the linear convergence rate is established in \cite{laws:1961,clin:1972}.  Several extensions of Lawson's iteration for the complex case as well as the new updating schemes  for the weight sequence $\{w_j^{(k)}\}$ are developed in \cite{elwi:1976} and \cite{rius:1968}, respectively. Moreover, recent applications of Lawson's iteration have also been presented  in \cite{fint:2018,nase:2018,natr:2020} for the best real and/or complex minimax rational approximations.

It has been known \cite{laws:1961,clin:1972} that the convergence of Lawson's iteration is linear and the convergence factor depends on how fast the weights corresponding to non-reference points decay to zero (see also our discussion in section \ref{subsec_revistLawson}). In practice, however, when the nodes are relatively dense in $\Omega$, or, there are too many nodes near some reference points, the convergence can be very slow, which then reduces its efficiency and accuracy as too many weighted least-squares problems need to be solved. Strategies for accelerating the convergence of the weights have been developed, e.g., in \cite{elwi:1976} and \cite{rius:1968} in the framework of Lawson's iteration, but some of them are not guaranteed for global convergence and the convergence rate is still linear. We will revisit Lawson's iteration in section \ref{subsec_revistLawson}.

In order to accelerate the convergence of weights and understand Lawson's iteration better, we made the following contributions in this paper:
\begin{itemize}
\item[(1)] By relying upon the Lagrange duality in modern optimization, we formulate an $L_q$-weighted dual programming of best Chebyshev problem \eqref{eq:bestf}; 
\item[(2)] By using the duality theory of  linear programming, for the real case, we provide an elementary and self-contained proof for the well-known Alternation Theorem;
\item[(3)] In the framework of $L_q$-weighted dual programming, we revisit Lawson's iteration and its convergence, and moreover, we propose an interior-point method for \eqref{eq:bestf}; the updating for weights in the interior-point method  is based on the Newton iteration and the convergence is much faster than Lawson's iteration. Numerical examples are reported and show that  the exactly $n+1$ reference points (i.e., equioscillation points in the real case) can be obtained in many cases.
\end{itemize}

The rest of the paper is organized as follows. In section \ref{sec_dual}, we establish the $L_q$-weighted dual programming of the best Chebyshev problem \eqref{eq:bestf} by the Lagrange duality theory. Based on an  equivalent  reformulation, for the real case, in section \ref{subsec_L1}, we provide an elementary and self-contained proof for the well-known Alternation Theorem relying on the duality theory in linear programming. In section \ref{subsec_dualL2}, we focus on the $L_2$-weighted dual programming of \eqref{eq:bestf}; in particular, we will present the gradient and the Hessian of the $L_2$-weighted dual objective function and also revisit Lawson's iteration in this framework. The new interior-point method for \eqref{eq:bestf} is presented in section \ref{sec:IP} and numerical experiments are reported in section \ref{sec:IP}.  Finally concluding remarks are drawn in section \ref{sec:conclusion}.

{\bf Notation}.
Throughout this paper, $\textsf{i}=\sqrt{-1}$ represents the imaginary unit, and for $\xi\in \bbC$, we write $\xi= \xi^{\tt r}+ {\tt i} \xi^{\tt i}$ and $\bar\xi= \xi^{\tt r}- {\tt i} \xi^{\tt i}$ where ${\rm Re }(\xi)=\xi^{\tt r}\in \bbR,~{\rm Im }(\xi)=\xi^{\tt i}\in \bbR$ and $|\xi|=\sqrt{(\xi^{\tt r})^2+(\xi^{\tt i})^2}$. Vectors are denoted in bold lower case letters, and ${\mathbb C}^{n\times m}$ (resp. ${\mathbb R}^{n\times m}$) is the set
of all $n\times m$ complex (resp. real) matrices, with $I_n\equiv [\be_1,\be_2,\dots,\be_n]\in\bbR^{n\times n}$ presenting the $n$-by-$ n$ identity matrix, where $\be_i$ is its $i$th column. 
For  $A\in\bbC^{m\times n}$, $A^{\HH}$ (resp. $A^{\T}$)  and $A^{\dag}$ are the conjugate transpose (resp. transpose) and the Moore-Penrose inverse of $A$, respectively; ${\rm span}(A)$ represents the column space of $A$ and $\|\bx\|_q=(\sum_{j=1}^n|x_j|^q)^\frac1q$ denotes the vector $q$-norm ($q\ge 1$) for $\bx\in \bbC^n$. Also, we denote the  $k$th Krylov subspace generated by a matrix $A$ on a vector $\bb$ by
\[
{\cal K}_k(A,\bb)={\rm span}(\bb,A\bb,\dots,A^{k-1}\bb)
\]
and  define the probability simplex  ($\be=[1,1,\dots,1]^{\T}$)
\begin{equation}\label{eq:simplex}
S:=\{\bw=[w_1,\dots,w_m]^{\T}\in \bbR^m: \bw\ge 0 ~{\rm and } ~\bw^{\T}\be=1\}.
\end{equation}
\section{The $L_q$-weighted dual of best Chebyshev problem} \label{sec_dual}
For a given $p\in \bbP_n$, define $r_i(p)=f_i-p(x_i)$ and 
$$
\br(p)=[r_1(p),\dots,r_m(p)]^{\T}\in \bbC^m, 
$$
and when there is no confusion can arise, we will simply write $\br = \br(p)$. 
We can first establish the $L_q$-weighted  dual programming.
\begin{theorem}\label{thm:q-dual}
Let $\{\psi_j\}_{j=1}^n$ be a basis of $\bbP_n$ satisfying the Haar condition on $\Omega$. Given $m\ge n+1$ distinct nodes $X=\{x_1,\dots,x_m\}$ on $\Omega$, let $p^*\in \bbP_n$ be the solution to \eqref{eq:bestf} and 
$$
\eta_\infty^*=\|\Bf-p^*(\bx)\|_\infty>0.
$$
Then  we have the following:
\begin{itemize}
\item[(1).]For $q\ge1$, 
\begin{equation}\label{eq:q-dual}
\mbox{($L_q$-weighted  dual programming)}\quad  \eta_\infty^*=\left(\max_{\bw\in S}d_q(\bw)\right)^{1/q}, 
\end{equation}
where for $\forall \bw\in S$, 
\begin{equation}\label{eq:q-dual2}
d_q(\bw)=\min_{p\in \bbP_n}\sum_{j=1}^m w_j |f_j-p(x_j)|^q;
\end{equation}
moreover, for any solution $\bw^*=[w_1^*,\dots,w_m^*]^{\T}$ of \eqref{eq:q-dual}, there are at least $n+1$ non-zero weights and  
$$\mbox{(complementary slackness)}\quad w_j^*(\eta_\infty^*-|f_j-p^*(x_j)|)=0, \quad\forall j=1,2,\dots,m.$$
\item[(2).] For $q>1$, then for any solution $\bw^*$ of \eqref{eq:q-dual}, $p^*$ is the unique solution for \eqref{eq:q-dual2} to achieve the minimum $d_q(\bw^*)=(\eta_\infty^*)^q$.
\end{itemize}
\end{theorem}
\begin{proof}
For (1), we can easily see that $(\eta_\infty^*)^q$ is the minimum of the following primal problem ($q\ge 1$)
\begin{equation}\label{eq:q-equv}
\min_{\eta,~p\in \bbP_n}\eta,~~s.t.,~ |f_j-p(x_j)|^q\le \eta, ~~\forall j=1,2,\dots,m.
\end{equation}
Defining the basis matrix 
\begin{equation}\nonumber
\Psi=\Psi(x_1,\dots,x_m):=\left[\begin{array}{cccc}\psi_1(x_1) & \psi_2(x_1) & \cdots & \psi_n(x_1) \\\psi_1(x_2) & \psi_2(x_2) & \cdots & \psi_n(x_2)  \\ \vdots & \cdots & \cdots& \vdots  \\\psi_1(x_m) & \psi_2(x_m) & \cdots & \psi_n(x_m) \end{array}\right]\in \bbC^{m\times n} \mbox{~with ~} \Psi_{i,j}=\psi_j(x_i),
\end{equation}
 for any $p\in \bbP_n$, we can parameterize the vector $p(\bx)\in \bbC^n$ as 
$p(\bx)=\Psi \ba$ associated with the coefficient vector $\ba=[a_1,\dots,a_n]^{\T}\in \bbR^n.$ We now consider the Lagrange dual problem of \eqref{eq:q-equv}. 

To apply the Lagrange dual theory, we write $f_j$, $\ba$ and $\Psi$ as the real and imaginary part and have 
\begin{align}\nonumber
 |f_j-\be_j^{\T}\Psi \ba|^q&=|f_j^{\texttt{r}}+\texttt{i} f_j^{\texttt{i}}-\be_j^{\T}(\Psi^{\tt r}+{\tt{i}} \Psi^{\tt i}) (\ba^{\tt r}+{\tt{i}} \ba^{\tt i})|^q\\\label{eq:jthresd1}
 &=\left((f_j^{\texttt{r}}-\be_j^{\T}\hat \Psi \hat\ba)^2 +(f_j^{\texttt{i}}-\be_j^{\T}\check \Psi \hat\ba)^2  \right)^{\frac q2} \\\label{eq:jthresd2}
 &=\left\|\hat {\Bf}_j-\left[\begin{array}{c}\be_j^{\T}\hat \Psi \\\be_j^{\T}\check \Psi \end{array}\right]\hat \ba\right\|_2^q,
\end{align}
where $\hat {\Bf}_j=\left[\begin{array}{c}f_j^{\tt r} \\f_j^{\tt i} \end{array}\right]\in \bbR^{2},~ \ba=\left[\begin{array}{c}\ba^{\tt r} \\\ba^{\tt i} \end{array}\right]\in \bbR^{2n}$, $\hat  \Psi=[\Psi^{\tt r}, -\Psi^{\tt i}]\in \bbR^{m\times 2n}$ and $\check  \Psi=[\Psi^{\tt i}, \Psi^{\tt r}]\in \bbR^{m\times 2n}$.  
Introducing the dual variable $w_j\in \bbR$ for each constraint in \eqref{eq:q-equv} and using \eqref{eq:jthresd1}, we have the Lagrange  function 
\begin{align*}
\ell(\eta,\hat\ba,\bw)&=\eta-\sum_{j=1}^m w_j( \eta-|f_j-\be_j^{\T}\Psi \ba|^q)\\
&=(1-\bw^{\T}{\be})\eta+\sum_{j=1}^m w_j \left((f_j^{\texttt{r}}-\be_j^{\T}\hat \Psi \hat\ba)^2 +(f_j^{\texttt{i}}-\be_j^{\T}\check \Psi \hat\ba)^2  \right)^{\frac q2}.
\end{align*}
Note that the Lagrange dual objective function is given by (see e.g., \cite[Chapter 12]{nowr:2006} or \cite{boyd:2004})
\begin{equation}\nonumber
d_q(\bw)=\inf_{\eta,~\hat\ba}\ell(\eta,\hat\ba,\bw).
\end{equation}
The domain ${\cal W}$ of $d_q(\bw)$ is ${\cal W}:=\{\bw:d_q(\bw)>-\infty\}$ and by \eqref{eq:q-dual}, we know ${\cal W}$ is just the probability simplex $S$ given by \eqref{eq:simplex}. Thus, 
\begin{equation}\nonumber
d_q(\bw)=\inf_{\eta,\hat\ba}\ell(\eta,\hat\ba,\bw)=\min_{\ba\in \bbC^n}\sum_{j=1}^m w_j |f_j-\be_j^{\T}\Psi \ba|^q.
\end{equation}
Since for $q\ge 1$,  by \eqref{eq:jthresd2}, we know that each term 
$$w_j |f_j-\be_j^{\T}\Psi \ba|^q=w_j\left\|\hat {\Bf}_j-\left[\begin{array}{c}\be_j^{\T}\hat \Psi \\\be_j^{\T}\check \Psi \end{array}\right]\hat \ba\right\|_2^q
$$ is convex \cite[Example 3.14]{boyd:2004} with respect to $\hat\ba$ for $\bw\in S$, and therefore, if we let $\hat \ba=\hat\ba(\bw)\in \bbR^{2n}$ be a solution (equivalently, we have $\ba(\bw)\in \bbC^{n}$) for the minimization  in \eqref{eq:q-dual2}, we can cast the Lagrange dual problem of \eqref{eq:q-equv} as 
\begin{equation}\nonumber
\max_{\bw\in S}d_q(\bw)=\max_{\bw\in S}\sum_{j=1}^m w_j |f_j-\be_j^{\T}\Psi \ba(\bw)|^q=\max_{\bw\in S}\min_{\ba\in \bbC^n}\sum_{j=1}^m w_j |f_j-\be_j^{\T}\Psi \ba|^q.
\end{equation}
The weak Lagrange duality ensures that for any $\bw\in S$,   $d_q(\bw)\le (\eta_\infty^*)^q$, and thus $$d_q^*:=\max_{\bw\in S}d_q(\bw)\le (\eta_\infty^*)^q.$$
For the strong duality $d_q^*=(\eta_\infty^*)^q$, we note that the Slater condition (see e.g., \cite[Section 5.2.3]{boyd:2004}) holds for \eqref{eq:q-equv} because there is $(\eta, \hat\ba)$ so that $\left((f_j^{\texttt{r}}-\be_j^{\T}\hat \Psi \hat\ba)^2 +(f_j^{\texttt{i}}-\be_j^{\T}\check \Psi \hat\ba)^2  \right)^{\frac q2} =|f_j-\be_j^{\T}\Psi \ba|^q<\eta$ for any $j=1,2,\dots,m$.

Let  $\bw^*$ be any solution of \eqref{eq:q-dual} and $p_d^*$ be any solution for \eqref{eq:q-dual2} to achieve $d_q(\bw^*)$.  From
\begin{align*}
\eta_\infty^*=(d_q^*)^{1/q}&=\left(\sum_{j=1}^m w^*_j|f_j-p_d^*(x_j)|^q\right)^{1/q}\\
&\le \left(\sum_{j=1}^m w^*_j|f_j-p^*(x_j)|^q\right)^{1/q}\\
&\le \left(\sum_{j=1}^m w^*_j (\eta_\infty^*)^q\right)^{1/q}=\eta_\infty^*,
\end{align*}
where the first inequality follows from the fact that $p_d^*$ is the minimizer for \eqref{eq:q-dual2} associated with $\bw^*$ and the second inequality from $|f_j-p^*(x_j)|\le \max_{1\le j\le m}|f_j-p^*(x_j)|=\eta_\infty^*$, we know that at any node $x_j$ with a nonzero $w_j^*\ne 0$, we have $|f_j-p^*(x_j)|=\eta_\infty^*$.  Thus we have the complementary slackness. 
We next show that there are at least $n+1$ nonzero weights $w_j^*$. Suppose  by contradiction that  there are  $k\le n$ nonzero weights, say, $w_1^*,\dots,w_k^*$. Then as $\{\psi_i\}_{i=1}^n$ satisfies the Haar condition, there is a vector $\ba\in \bbC^n$ so that 
$f_j=\be_j^{\T} \Psi(x_1,\dots,x_k)\ba$ for $1\le j\le k$ because  $\rank(\Psi(x_1,\dots,x_k))=k$; that is, there is a $p\in \bbP_n$ so that the minimizer $p\in \bbP_n$ of \eqref{eq:q-dual2} associated with $\bw^*$  satisfies $d_q(\bw^*)=d_q^*=\min_{p\in \bbP_n}\sum_{j=1}^m w_j^* |f_j-p(x_j)|^q=0$, a contradiction.

For (2), suppose, without loss of generality,  that $w_1^*,\dots,w_k^*$ are nonzero for $k\ge n+1$ and we show that 
$$\min_{p\in \bbP_n}\sum_{j=1}^m w_j^* |f_j-p(x_j)|^q=\min_{p\in \bbP_n}\sum_{j=1}^k w_j^* |f_j-p(x_j)|^q$$ has a unique solution $p^*$ by the strict convexity of this minimization. Indeed, 
 with $W_k^*=\diag(w_1^*,\dots,w_k^*)>0$, $\bk=[f_1,\dots,f_k]^{\T}\in \bbC^k$ and \eqref{eq:jthresd2}, let 
 \begin{align*}
 m(\ba)&:=\left(\sum_{j=1}^k w_j^* |f_j-\be_j^{\T}\Psi(x_1,\dots,x_k)\ba|^q\right)^{1/q}=\left\|(W_k^*)^{\frac1q}(\bk-\Psi(x_1,\dots,x_k) \ba)\right\|_{q}
\\
&=\sum_{j=1}^k w_j^*\left\|\hat {\Bf}_j-\left[\begin{array}{c}\be_j^{\T}\hat \Psi(x_1,\dots,x_k) \\\be_j^{\T}\check \Psi(x_1,\dots,x_k)\end{array}\right]\hat \ba\right\|_2^q=:\hat m(\hat \ba);
 \end{align*}
then we know that $\hat m(\hat\ba)$ is convex with respect to $\hat \ba$ (equivalently, $m(\ba)$ is convex with respect to $\ba$) due to the Minkowski inequality, i.e.,  $m(\alpha \ba+(1-\alpha)\bb)\le \alpha m( \ba)+(1-\alpha)m(\bb)$, $\forall 0\le \alpha\le 1$; moreover, when the equality holds, then 
 $\bk- \Psi(x_1,\dots,x_k) \ba=\xi\left( \bk- \Psi(x_1,\dots,x_k) \bb\right)$ for some $\xi\in\bbC$. If $\xi\ne 1$, then we have $$\bk= \Psi(x_1,\dots,x_k)\frac{\ba-\xi\bb}{1-\xi}$$  implying again that   $d_q(\bw^*)=d_q^*=\min_{p\in \bbP_n}\sum_{j=1}^m w_j^* |f_j-p(x_j)|^q=0$, a contradiction.  Hence $\Psi(x_1,\dots,x_k)(\ba-\bb)=0$ which by $\rank(\Psi(x_1,\dots,x_k))=n$ leads to $\ba=\bb$.  Consequently, the strict convexity of $m(\ba)$ holds which ensures that $p^*$ is the unique solution for \eqref{eq:q-dual2} to achieve the minimum $d_q(\bw^*)=(\eta_\infty^*)^q$.
 %
%
\end{proof}


\section{The Alternation Theorem for the real case}\label{subsec_L1}
It is well-known that for the real case, the primal programming \eqref{eq:q-equv} of the particular case with $q=1$ can be solved by a linear programming (LP). Indeed,  \eqref{eq:q-equv} can be reformlualted as 
\begin{equation}\nonumber
\min_{\eta,\ba\in \bbR^n}\eta, ~~s.t.,\ |\Bf-\Psi\ba|\le \eta \be 
\end{equation}
 which can be rewritten as the following LP \cite[Section 1.2.2]{boyd:2004}:
\begin{equation}\label{eq:L1LP-b}
\min_{\eta,\ba\in \bbR^n}  
[\mathbf{0}_{n}^{\T},1]\left[\begin{array}{c}\ba \\\eta\end{array}\right], ~~s.t., \left[\begin{array}{cc}\Psi & -\be\\-\Psi & -\be\end{array}\right]\left[\begin{array}{c}\ba \\\eta\end{array}\right]\le \left[\begin{array}{c}\Bf \\-\Bf\end{array}\right].
\end{equation}
In the case of $\bbP_n={\rm span}(1,x,\cdots,x^{n-1})$, a natural way of expressing $p^*=[1,x,\cdots,x^{n-1}]\ba$ is to use the basis $\{1,x,\cdots,x^{n-1}\}$, which leads to the ill-conditioned Vandermonde matrix 
\begin{equation}\label{eq:vadm}
V_{\bx}=[\be,X\be,\dots,X^{n-1}\be]=\left[\begin{array}{cccc}1 & x_1& \dots & x_1^{n-1} \\1 & x_2& \dots & x_2^{n-1} \\ & \dots & \dots &  \\1 & x_m& \dots & x_m^{n-1}\end{array}\right]
\end{equation}
for $\Psi$ in \eqref{eq:L1LP-b} and consequently an ill-conditioned LP. Fortunately, recent development in \cite{brnt:2021,zhsl:2023} implies that we can use the Lanczos process (see e.g., \cite{govl:2013}) to work on an orthogonal basis  $\Psi_{\bx}$ with $\Psi_{\bx}^{\T}\Psi_{\bx}=\sqrt{m}I_m$ and $V_{\bx}=\Psi_{\bx} R_{\bx}$ where $R_{\bx}$ is upper triangular, to write $V_{\bx}\ba=\Psi_{\bx}(R_{\bx}\ba)=\Psi_{\bx}\tilde\ba$ where $\tilde\ba=R_{\bx}\ba$. Thus, the resulting LP is 
\begin{equation}\label{eq:L1LP-c}
\min_{\eta,\tilde\ba\in \bbR^n}  
[\mathbf{0}_{n}^{\T},1]\left[\begin{array}{c}\tilde\ba \\\eta\end{array}\right], ~~s.t., \left[\begin{array}{cc}\Psi_{\bx} & -\be\\-\Psi_{\bx} & -\be\end{array}\right]\left[\begin{array}{c}\tilde\ba \\\eta\end{array}\right]\le \left[\begin{array}{c}\Bf \\-\Bf\end{array}\right].
\end{equation}
Particularly, we know that the coefficient matrix $\left[\begin{array}{cc}\Psi_{\bx} & -\be\\-\Psi_{\bx} & -\be\end{array}\right]$ of   the LP \eqref{eq:L1LP-c}  is orthogonal. Even though recovering the coefficient vector $\ba$ in the basis $\{1,x,\cdots,x^{n-1}\}$ still involves solving an ill-conditioned system $R_{\bx}\ba=\tilde \ba$, this can be avoided if we only need to compute  values of $p^*(\bv)$ at new nodes $\bv=[v_1,\dots,v_N]^{\T}\in\bbR^N$. Specifically, the Lanczos process can be proceeded reversely to have a basis $\Psi_{\bv}$ for the new Vandermonde
\begin{equation}\label{eq:vadms}
V_{\bv}=\left[\begin{array}{cccc}1 & v_1& \dots & v_1^{n-1} \\1 & v_2& \dots & v_2^{n-1} \\ & \dots & \dots &  \\1 & v_N& \dots & v_N^{n-1}\end{array}\right]
\end{equation}
satisfying $V_{\bv}=\Psi_{\bv} R_{\bx}$ (see \cite[Algorithm 1 and Figure 2.1]{zhsl:2023})  and, therefore, 
$$\bp_{\bv}=[p^*(v_1),\cdots,p^*(v_N)]^{\T}=\Psi_{\bv}\tilde \ba,$$ where $\tilde \ba$ is computed from \eqref{eq:L1LP-c}. For more detailed discussion, refer to \cite{brnt:2021,zhsl:2023}.

For the real case, as a by-product of the equivalent linear programming  formulation \eqref{eq:L1LP-c} for the original \eqref{eq:bestf}, we can provide an elementary proof for the well-known Alternation Theorem (aks the Chebyshev Equioscillation Theorem) (see e.g., \cite[Chapter 3.4]{chen:1982}) using basic results of LP and the following lemma (see \cite[Chapter 3.4]{chen:1982}).
\begin{lemma}\label{lem:genVand}
Let $\{\psi_i\}_{i=1}^n$ be a system of elements from continuous functions on $\Omega=[a,b]$ satisfying the Haar condition. Then for arbitrary $n+1$ points $\{z_j\}_{j=0}^{n}$ with $a\le z_0<z_1<\dots<z_n\le b$, the set $\{\bt\in\bbR^{n+1}:[\Psi(z_0,\dots,z_n)]^{\T}\bt=\mathbf{0}\}$ is a one dimensional subspace and for any nonzero solution $\bt=[t_1,\dots,t_{n+1}]^{\T}$, it holds $t_jt_{j+1}<0,~\forall j=1,2\dots, n$.
\end{lemma}
\begin{proof}
The result is from \cite[Chapter 3.4]{chen:1982} and $\rank(\Psi(z_0,\dots,z_n))=n$.
\end{proof}
\begin{theorem} {\rm (Alternation Theorem)}\label{thm:Alternation}
Let $\{\psi_i\}_{i=1}^n$ be a system of elements from continuous functions on $\Omega=[a,b]$  satisfying the Haar condition, and $X=\{x_1,\dots,x_m\}$ with $a\le x_1<x_2<\ldots<x_m\le b$ and $m\ge n+1$. Then $p^*$ is the unique solution to \eqref{eq:bestf} if and only if there are at least $n+1$ points $x_{\beta_1}<x_{\beta_2}<\dots<x_{\beta_{n+1}}$ in $X$ so that $p^*(x_{\beta_j})-f_{{\beta_j}}=(-1)^{j}\|p^*(\bx)-\Bf\|_{\infty}$ for $1\le j\le n+1$ or, $p^*(x_{\beta_j})-f_{{\beta_j}}=(-1)^{j+1}\|p^*(\bx)-\Bf\|_{\infty}$ for $1\le j\le n+1$.
\end{theorem}
\begin{proof}  
The result is trivial if $\eta_\infty^*=0$, i.e., $p^*$ is an interpolation of $f$ over $X$. For the general case $\eta_\infty^*>0$, we can write the dual LP of \eqref{eq:L1LP-b} as follows
 \begin{align}\label{eq:L1LP-b-daul}
\max_{\by_+,\by_- \in\bbR^m}&  \nonumber
-[\Bf^{\T},-\Bf^{\T}] \left[\begin{array}{c}\by_+ \\\by_-\end{array}\right]\\
 s.t., ~&\left[\begin{array}{cc}\Psi^{\T} & -\Psi^{\T}\\-\be^{\T} & -\be^{\T}\end{array}\right]\left[\begin{array}{c}\by_+ \\\by_-\end{array}\right]= \left[\begin{array}{c}\mathbf{0}_{n}\\-1\end{array}\right],\quad \left[\begin{array}{c}\by_+ \\\by_-\end{array}\right]\ge \mathbf{0}_{2m},
\end{align}
which is a standard LP. We shall establish the Alternation Theorem based on \eqref{eq:L1LP-b} and \eqref{eq:L1LP-b-daul}.

As the LP \eqref{eq:L1LP-b} of \eqref{eq:bestf} has optimal solutions (by the equivalence between \eqref{eq:bestf} and   \eqref{eq:L1LP-b}),  and the coefficient matrix of \eqref{eq:L1LP-b-daul} has rank $n+1$, the fundamental  theorem (see e.g., \cite[Theorems 13.1 and 13.2]{nowr:2006}) of LP implies that the maximum can be achieved in a {\it basic optimal solution} $\by^*=\left[\begin{array}{c}\by^*_+ \\\by^*_-\end{array}\right]$ with the associated basic variables $\{y_{\nu_1,\beta_1},y_{\nu_2,\beta_2},\dots,y_{\nu_{n+1},\beta_{n+1}}\}$ and $\nu_j\in \{+,-\}$. Let  ${\cal B}=\{\beta_1,\dots,\beta_{n+1}\}\subset \{1,2,\dots,m\}=:[m]$ and thus $y^*_j=0$ for all $j\in[m] \setminus {\cal B}$.  Denote by $$\bt^*=[t^*_1,\dots,t^*_{m}]^{\T}:=\by^*_+-\by^*_-=[y^*_{+,1}-y^*_{-,1},\dots,y^*_{+,m}-y^*_{-,m}]^{\T}.$$ According to the strong duality \cite[Theorem 13.1]{nowr:2006}, the maximum $-\Bf^{\T}\bt^*$ is equal to the minimum of \eqref{eq:L1LP-b}, which is $\eta_\infty^*=\|\Bf-p^*(\bx)\|_\infty$ as \eqref{eq:L1LP-b} is equivalent to \eqref{eq:bestf}; that is $-\Bf^{\T}\bt^*=\eta_\infty^*>0.$ 
By the first $n$ equality constraints in \eqref{eq:L1LP-b-daul}, we have  $\Psi^{\T}\bt^*=[\Psi(x_{\beta_1},\dots, x_{\beta_{n+1}})]^{\T}\bt_{\cal B}^*=0$, where we have used the fact that $t^*_j=0$ for all $j\in[m] \setminus {\cal B}$, and $\bt_{\cal B}^*=[t_{\beta_1}^*,\dots, t_{\beta_{n+1}}^*]^{\T}\in \bbR^{n+1}$ is a sub-vector of $\bt^*$. By  $-\Bf^{\T}\bt^*=\eta_\infty^*>0$, $\bt_{\cal B}^*\neq \mathbf{0}_{n+1}$, and thus we can apply Lemma \ref{lem:genVand} to conclude  $t^*_{\beta_j}t^*_{\beta_{j+1}}<0,~\forall j=1,2\dots, n$. Because $t_j^*:=y^*_{+,j}-y^*_{-,j}$ and there are at most $n+1$ nonzero elements in $\by^*$, we know that for each $t_{\beta_j}^*$,  one of $y^*_{+,{\beta_j}}$ and $y^*_{-,{\beta_j}}$ is positive and the other is zero. This also implies that every basic optimal solution $\by^*$ of \eqref{eq:L1LP-b-daul} is non-degenerate \cite[Definition 13.1]{nowr:2006}. 
On the other hand, according to the complementarity condition (see e.g., \cite[Equ. (13.9d)]{nowr:2006}), we have $t^*_{\beta_j}(|f_{\beta_j}-p^*(x_{\beta_j})|-\eta_{\infty}^*)=0$ and the alternation property $t^*_{\beta_j}t^*_{\beta_{j+1}}<0,~\forall j=1,2\dots, n$ implies either  $p^*(x_{\beta_j})-f_{{\beta_j}}=(-1)^{j}\eta_\infty^* $ for $1\le j\le n+1$ or, $p^*(x_{\beta_j})-f_{{\beta_j}}=(-1)^{j+1}\eta_\infty^*$ for $1\le j\le n+1$.

We next show that if $p^*\in \bbP_n$ satisfies the alternation property, then $p^*$ is a solution to \eqref{eq:bestf}. Without loss of generality, for the first $n+1$ points $x_{1}<x_{2}<\dots<x_{n+1}$ in $X$ (i.e., ${\cal B}=\{1,2,\dots,n+1\}$), assume   $p^*(x_{j})-f_{{j}}=(-1)^{j-1}\|p^*(\bx)-\Bf\|_{\infty}$ for $1\le j\le n+1$. Our idea is to construct the Karush–Kuhn–Tucker (KKT) conditions (see e.g., \cite[Equ. (14.3)]{nowr:2006}) and show that $p^*$ corresponds to a  basic optimal solution for  \eqref{eq:L1LP-b}; as \eqref{eq:L1LP-b} is equivalent to the original \eqref{eq:bestf}, we then know that $p^*$ is the solution to \eqref{eq:bestf}. To this end, let $p^*(\bx)=\Psi \ba^*$ and  introduce the slack variables $\bs\ge \mathbf{0}$ and write \eqref{eq:L1LP-b} as 
\begin{equation}\label{eq:L1LP-bs}
\min_{\eta,\ba\in \bbR^n}  
[\mathbf{0}_{n}^{\T},1]\left[\begin{array}{c}\ba \\\eta\end{array}\right], ~~s.t., \left[\begin{array}{cc}\Psi & -\be\\-\Psi & -\be\end{array}\right]\left[\begin{array}{c}\ba \\\eta\end{array}\right]+\bs= \left[\begin{array}{c}\Bf \\-\Bf\end{array}\right],~~\bs\ge \mathbf{0}.
\end{equation}
Define $\eta^*=\|p^*(\bx)-\Bf\|_{\infty}$ and $$\bs^*=[s^*_{+,1},\dots,s^*_{+,m},s^*_{-,1},\dots,s^*_{-,m}]^{\T}:=\left[\begin{array}{c}\Bf \\-\Bf\end{array}\right]-\left[\begin{array}{cc}\Psi & -\be\\-\Psi & -\be\end{array}\right]\left[\begin{array}{c}\ba^* \\\eta^*\end{array}\right];$$ 
by the alternation property, we know  $s_{+,j}^*=0$ for all odd $1\le j\le n+1$ and $s_{-,j}^*=0$ for all even $1\le j\le n+1$, and  other $s_{\pm,j}^*\ge 0$. We next show $(\ba^*,\eta^*,\bs^*)$ is an optimal solution for \eqref{eq:L1LP-bs} and thus $\eta^*=\eta_\infty^*$ and $p^*$ solves \eqref{eq:bestf}. 

According to $\bs^*$, we construct a feasible solution $\by^*=\left[\begin{array}{c}\by^*_+ \\\by^*_-\end{array}\right]$ of \eqref{eq:L1LP-b-daul} satisfying the KKT conditions \cite[Equ. (14.3)]{nowr:2006}. First, let $y_{+,j}^*=y_{-,j}^*=0$ for $n+2\le j\le m$ and thus $t^*_j:=y^*_{+,j}-y^*_{-,j}=0$ $(n+2\le j\le m)$. For $\bt^*_{\cal B}=[t^*_1,\dots,t^*_{n+1}]^{\T}$, by Lemma \ref{lem:genVand}, we can choose a nonzero solution $\bt^*_{\cal B}$ for $[\Psi(x_{1},\dots, x_{{n+1}})]^{\T}\bt_{\cal B}^*=\Psi^{\T}\bt^*=\mathbf{0}_{n+1}$ so that $ t^*_{1}>0$, $t^*_{j}t^*_{{j+1}}<0,~\forall j=1,2\dots, n$ and $\sum_{j=1}^{n+1}(-1)^{j-1}t_j=1$. Then for $1\le j\le n+1$, 
define
\begin{eqnarray}\nonumber
 y_{+,j}^*=\left\{
\begin{array}{ccc} t_j^*,& j \mbox{ is odd},\\0,& j \mbox{ is even},
\end{array} \right.\quad 
y_{-,j}^*=\left\{
\begin{array}{ccc} 0,& j \mbox{ is odd},\\-t_j^*,& j \mbox{ is even}.
\end{array} \right.
\end{eqnarray}
It can be easily verified that such a $\by^*$ is feasible for \eqref{eq:L1LP-b-daul} and also satisfies the complementarity condition $y^*_{\nu,j}s^*_{\nu,j}=0~(1\le j\le  m,~\nu\in\{+,-\})$. Consequently, the KKT conditions \cite[Equ. (14.3)]{nowr:2006} are fulfilled, implying that $(\ba^*,\eta^*,\bs^*)$ is an optimal solution for \eqref{eq:L1LP-bs} and $p^*$ solves \eqref{eq:bestf}.
\end{proof}

\section{The $L_2$-weighted dual programming}\label{subsec_dualL2}
\subsection{The gradient and the Hessian of the dual objective function}\label{subsec_gradhess}
For the general complex case, another interesting case is $q=2$.  This leads to the $L_2$-weighted dual programming (we simply denote $d(\bw)=d_2(\bw)$ for the dual objective function  in \eqref{eq:q-dual2})
\begin{equation}\label{eq:dualmaxdw}
\max_{\bw\in S}d(\bw)
\end{equation}
and   Lawson's iteration can be analyzed on this dual problem.  In the following proposition, we provide the gradient and the Hessian of the Lagrange dual function $d(\bw)$ for $\bw>0$.

\begin{proposition}\label{prop:grad_dw}
For $q=2$, denote  the Lagrange dual function \eqref{eq:q-dual2} by 
\begin{equation}\label{eq:q-dual_q2}
d(\bw)=\min_{\ba\in \bbC^n}\|W^{\frac12}\Bf-W^{\frac12}\Psi \ba\|_2^2,
\end{equation}
where $W=\diag(w_1,\dots,w_m)$ with diagonal entries from $\bw$. Then for $\bw>0$, the gradient and the Hessian of $d(\bw)$ are given by 
\begin{equation}\label{eq:grad_dw}
\nabla d(\bw)=\left[\begin{array}{c}|r_1|^2 \\\vdots \\|r_m|^2\end{array}\right]\in \bbR^m,
\end{equation}
and 
\begin{equation}\label{eq:hess_dw}
\nabla^2 d(\bw)=-2{\rm Re}\left(\diag(\br)^{\HH} \Psi(\Psi^{\HH}W\Psi)^{-1}\Psi^{\HH}\diag(\br)\right)\in \bbR^{m\times m},
\end{equation}
where $\br=[r_1,\dots,r_m]^{\T}\in \bbC^n$ with $r_i=f_i-\be_i^{\T}\Psi \ba(\bw)\in \bbC$ and $\ba(\bw)\in \bbC^n$ is the solution for the minimization of \eqref{eq:q-dual_q2}. 
\end{proposition}
\begin{proof}
For the solution  $\ba(\bw)$, by $\bs=W^{\frac12}\Bf-W^{\frac12}\Psi \ba(\bw)$,  the conclusion follows by calculations. In particular, we have 
\begin{equation}\label{eq:grad_d}
\nabla d(\bw) =2{\rm Re}\left(\left(\frac{\partial \bs}{\partial \bw}\right)^{\HH}\bs\right)=\left[\begin{array}{c}|r_1|^2 \\\vdots \\|r_m|^2\end{array}\right]-2{\rm Re}\left(\left(\frac{\partial \ba(\bw)}{\partial \bw}\right)^{\HH}\Psi^{\HH} W \br\right).
\end{equation}
Note that  $\ba(\bw)$ is the solution to the least-squares problem \eqref{eq:q-dual_q2} and using the pseudo-inverse $(W^{\frac12}\Psi)^{\dag}$, it admits the closed form as 
$\ba(\bw)=(W^{\frac12}\Psi)^{\dag}W^{\frac12}\Bf$, or equivalently, $\Psi^{\HH}W\Psi\ba(\bw)=\Psi^{\HH}W\Bf$.
Furthermore, because $W^{\frac12}\br$ is just the residual of the least-square problem \eqref{eq:q-dual_q2}, the optimal condition of \eqref{eq:q-dual_q2} implies that $W^{\frac12}\br\perp W^{\frac12}\Psi$ or equivalently $\Psi^{\HH}W\br=\mathbf{0}$ and thus \eqref{eq:grad_dw} follows.

Finally, by computing 
\begin{align*}
[\nabla^2 d(\bw)]_{i,j}&=\frac{\partial |r_i(\bw)|^2}{\partial w_j}=2{\rm Re}\left(\frac{\partial \bar r_i(\bw)  r_i}{\partial w_j}\right)
=-2{\rm Re}\left(\bar r_i\be_i^{\T}\Psi (\Psi^{\HH}W\Psi)^{-1}\Psi^{\HH}\be_jr_j\right),
\end{align*}
\eqref{eq:hess_dw} holds.
\end{proof}

\begin{remark}\label{rmk:orthPsi}
For the Hessian of $d(\bw)$, we notice that if   $W^{\frac12}\Psi=\Psi_WR_W ~(\Psi_W\in \bbC^{m\times n},~R_W\in \bbC^{n\times n})$ is the thin QR decomposition of $W^{\frac12}\Psi$, then we have
\begin{equation}\label{eq:hess_dw_orth} 
\nabla^2 d(\bw)=-2{\rm Re}\left(R^{\HH}W^{-\frac12} \Psi_W\Psi_W^{\HH}W^{-\frac12}R\right),~ ~\mbox{with~~} R=\diag(\br)\in \bbC^{n\times n}.
\end{equation}
In the case of $\bbP_n={\rm span}(1,x,\cdots,x^{n-1})$, we will see in section \ref{subsec_revistLawson} that computing \eqref{eq:hess_dw_orth} can be done in $O(mn)$ flops by using the Arnoldi process.
\end{remark}

\subsection{Revisit Lawson's iteration}\label{subsec_revistLawson}
Note that the primary computational task in each Lawson's iteration Algorithm \ref{alg:lawson} is to solve a weighted least-squares problem \eqref{eq:q-dual_q2}. We remark that during the iteration, it is not the associated coefficient vector $\ba^{(k)}$ in $\bp^{(k)}=\Psi\ba^{(k)}$ but the residual $\br^{(k)}$ that should be computed accurately and efficiently in order to update the weight vector $\bw^{(k)}$ for the next step. As the weighted matrix $W^{(k)}=\diag(\bw^{(k)})$ is diagonal, accurately computing $W^{(k)}\br^{(k)}$ is sufficient to determine $\br^{(k)}$. 

In the case of $\bbP_n={\rm span}(1,x,\cdots,x^{n-1})$ where $\Psi=V_{\bx}$ in \eqref{eq:vadm}, a recent development in \cite{brnt:2021,zhsl:2023} provides an effective and efficient way for computing $\br^{(k)}$ without involving explicitly the Vandermonde matrix $V_{\bx}$. In particular, noticing  
$$
\sqrt{W^{(k)}}V_{\bx}=[\sqrt{\bw^{(k)}},X\sqrt{\bw^{(k)}},\dots,X^{n-1}\sqrt{\bw^{(k)}}],
$$
we can apply the Arnoldi process to generate the thin QR factorization $\sqrt{W^{(k)}}V_{\bx}=\Psi_{W^{(k)}}R_{W^{(k)}} $ of $\sqrt{W^{(k)}}V_{\bx}$, where the orthonormal matrix $\Psi_{W^{(k)}}$ is also the basis for the Krylov subspace ${\cal K}_{n}(\sqrt{\bw^{(k)}}, X)$ (see \cite[Theorem 2.1]{zhsl:2023}); the computational flops for forming $\Psi_{W^{(k)}}$ in the Arnoldi process is $O(mn)$ as $R_{W^{(k)}}$ is not explicitly used. 
Thus with the new vector $\tilde \ba=R_{W^{(k)}}\ba$, it holds that 
\begin{equation}\label{eq:LancWr}
\|\sqrt{W^{(k)}}\br^{(k)}\|_2^2=
\min_{\ba\in \bbC^n}\|\sqrt{W^{(k)}}\Bf-\sqrt{W^{(k)}}V_{\bx} \ba\|_2^2=\min_{\tilde\ba\in \bbC^n}\|\sqrt{W^{(k)}}\Bf-\Psi_{W^{(k)}}\tilde\ba\|_2^2
\end{equation}
where the optimal\footnote{In MATLAB, we can simply compute $\tilde\ba$ by the backslash: $\tilde\ba=\Psi_{W^{(k)}} \backslash(\sqrt{W^{(k)}}*\Bf)$.}
 $\tilde\ba^{(k)}=\Psi_{W^{(k)}}^\dag \sqrt{W^{(k)}}\Bf$  and thus the optimal residual of \eqref{eq:q-dual_q2} is 
\begin{equation}\label{eq:LancWr2}
\sqrt{W^{(k)}}\br^{(k)}= \sqrt{W^{(k)}}\Bf-\Psi_{W^{(k)}}\tilde\ba^{(k)}.
\end{equation}
With this technique for updating the weights, suppose $W^*=\diag(\bw^*)$ is the convergent weight matrix for \eqref{eq:dualmaxdw} and $\tilde \ba^*$ is the associated optimal solution related with \eqref{eq:LancWr} (correspondingly, we have $\sqrt{W^*}V_{\bx}=\Psi_{W^*}R_{W^*}$). Analogous to our discussion in section \ref{subsec_L1}, for new nodes $\{v_j\}_{j=1}^N\subseteq \Omega$, we can compute the new values $\bp_{\bv}=[p^*(v_1),\cdots,p^*(v_N)]^{\T}$ of the optimal approximation $p^*$ as $\bp_{\bv}=\Psi_{\bv}\tilde \ba^*$, where $\Psi_{\bv}$ is obtained by reversing the Arnoldi process (see \cite[Algorithm 1 and Figure 2.1]{zhsl:2023}) and is related with the Vandermonde matrix $V_{\bv}$ in \eqref{eq:vadms} via   $V_{\bv}=\Psi_{\bv} R_{W^*}$. For more detailed discussion of this issue, refer to \cite{brnt:2021,zhsl:2023}.

From the complementary slackness $w_i^*(|f_i-p^*(x_i)|-\eta_\infty^*)=0$ for all $1\le i\le m$ in Theorem \ref{thm:q-dual}, we can partition the nodes set $X$ as  
\begin{equation}\nonumber
X={\cal I}\bigcup {\cal E},~~{\cal I}\bigcap {\cal E}=\emptyset,
\end{equation}
where 
\begin{equation}\label{eq:AJ}
{\cal E}=\{x_i\in X:|f_i-p^*(x_i)|=\eta_\infty^*\}\quad \mbox{and}\quad {\cal I}=\{x_i\in X:|f_i-p^*(x_i)|<\eta_\infty^*\}.
\end{equation}
The subsets ${\cal E}$ and ${\cal I}$ are unknown a prior,  and Lawson's iteration \cite{laws:1961} and/or its restarted version \cite{elwi:1976} aims at computing $\{w_i^*\}$ and detecting a subset ${\cal \hat E}$ of ${\cal E}$ on which $p^*$ is also the best minimax approximation to $\Bf$. Specifically, in the original Lawson's iteration\cite{laws:1961}, the weights $w_i^{(k)}$ are updated according to step 3 of   Algorithm \ref{alg:lawson}. Define  subsets ${X}_1,X_2$ and ${X}_3$ of $[m]$ as 
\begin{align}\nonumber
{X}_1&=\{j: w_j^{(k)}\ge \nu>0 \mbox{ for infinitely many $k$ and some } \nu\},\\\nonumber
{X}_2&=\{j: \lim_{k\rightarrow \infty}w_j^{(k)}=0 \mbox{ and } w_j^{(k)}>0,~\forall k\},\\\nonumber
{X}_3&=\{j: j\in {\cal I} \mbox{ and } w_j^{(k)}>0,~\forall k\}.
\end{align}
It is shown \cite{laws:1961,clin:1972,rice:1969} that 
\begin{itemize}
\item[(i).] $\{d(\bw^{(k)})\}_{k=1}^\infty$ is monotonically increasing and convergent;
\item[(ii).]  $\lim_{k\rightarrow \infty}|w_i^{(k)}-w_i^{(k+1)}|=0$ for any $1\le i\le m$;
\item[(iii).] $p^{(k)}$  converges to a function $\tilde p\in \bbP_n$ which is the best Chebyshev approximation of $f(x)$ on ${X}_1\cup {X}_2$;
\item[(iv).] if $\tilde p=p^*$, then the convergence is linear with the linear convergence factor \cite{clin:1972}
\begin{equation}\label{eq:rho}
\rho=\frac{\max_{i\in X_3} |r_i^*|}{\eta^*_\infty}<1.
\end{equation}
\end{itemize}

Lawson's updating scheme generally is able to find the optimal $p^*$, but counterexamples exist in which the convergent $\tilde p\neq p^*$ (see e.g., \cite[Section 13-11]{rice:1969}). This mis-convergence happens when some critical points in ${\cal E}$  that determine the optimal $p^*$ are ``accidentally" but permanently set to zeros.  Restarting the weights is a possible remedy for this mis-convergence.
Also, even if  convergence to $p^*$ is guaranteed, the convergence can be very slow. This can be seen from the ratio $\rho$ \eqref{eq:rho} which is the largest non-maximum error of $p^*$ to the maximum error $\eta^*_\infty$. In the case when the nodes are relatively dense in  $\Omega$, or, there are too many nodes near some critical points 
in ${\cal E}$  that determine the optimal $p^*$, by the continuity of the optimal residual $\br(x)$, the numerator $\max_{i\in X_3} |r_i^*|$ of \eqref{eq:rho} can be sufficiently close to $\eta_\infty^*$, leading to a very slow convergence of the weights. One traditional remedy \cite{laws:1961,clin:1972,rice:1969}  is to set  certain weights to be zero as long as they are within some threshold $\epsilon_w$, which, however, is not guaranteed as a proper not sufficiently small threshold tolerance $\epsilon_w$ is also problem-dependent. Other remedies include modify the scheme  for updating the weights. For example, another simple updating formulation \cite{elwi:1976,rius:1968,rice:1969} is 
\begin{equation}\nonumber
w_j^{(k+1)}=\frac{w_j^{(k)}|r_j^{(k)}|^2}{\sum_{i=1}^m w_i^{(k)} |r_i^{(k)}|^2},  ~~\forall j=1,2,\dots,m,
\end{equation} 
or equivalently, 
$$
\bw^{(k+1)}=\frac{\bw^{(k)}\odot \nabla d(\bw^{(k)})}{\be^{\T}\left(\bw^{(k)}\odot \nabla d(\bw^{(k)})\right)}=\frac{\bw^{(k)}\odot \nabla d(\bw^{(k)})}{d(\bw^{(k)})},
$$
where $\odot$ is the Hadamard product and $d(\bw^{(k)})$ is the objective value of the dual in \eqref{eq:q-dual_q2}.
It is known that the sequence $\{d(\bw^{(k)})\}_{k=1}^\infty$ is also monotonically increasing \cite{elwi:1976} and convergent because 
\begin{align*}
d(\bw^{(k+1)})&=\frac{1}{d(\bw^{(k)})}\sum_{i=1}^m\left|\sqrt{w_i^{(k)}}\bar r_i^{(k)} r_i^{(k+1)}\right|^2\cdot \sum_{i=1}^m\left(\sqrt{w_i^{(k)}}\right)^2\\
&\ge \frac{1}{d(\bw^{(k)})}\left|\sum_{i=1}^mw_i^{(k)}\bar r_i^{(k)} r_i^{(k+1)}\right|^2\quad \quad \mbox{(by Cauchy-Schwartz inequality)}\\
&=\frac{1}{d(\bw^{(k)})}\left|\sum_{i=1}^mw_i^{(k)}|r_i^{(k)}|^2+\sum_{i=1}^mw_i^{(k)} \bar r_i^{(k)} (r_i^{(k+1)}-r_i^{(k)})\right|^2\\
&= \frac{1}{d(\bw^{(k)})}\left(\sum_{i=1}^mw_i^{(k)}|r_i^{(k)}|^2\right)^2~  \mbox{(by optimality $\sqrt{W^{(k)}}\br^{(k)}\perp \sqrt{W^{(k)}}\Psi$ of \eqref{eq:LancWr})}\\
&=d(\bw^{(k)}).
\end{align*}
However, different from $q=1$, the convergence of $p^{(k)}$ is not guaranteed \cite{laws:1961,rice:1969}. Indeed, whenever $d(\bw^{(k+1)})=d(\bw^{(k)})$, the Cauchy-Schwartz inequality ensures only 
$$
 \bar r_i^{(k)}  r_i^{(k+1)}=\mbox{constant}, \quad \forall i\in \{i:w_i^{(k)}>0\},
$$ 
and $\{|r_i^{(k)}|\}_{k=1}^\infty$ may behave  oscillatingly. As an illustration, in the following example, we demonstrate  behaviors of Lawson's iteration with the two updating schemes.

\begin{example}\label{eg:1}
We apply Lawson's iteration with updating $w_j^{(k+1)}=\frac{w_j^{(k)} |r_j^{(k)}|^q}{\sum_{i=1}^m w_i^{(k)} |r_i^{(k)}|^q}$ for $q=1$ and $q=2$ in the following test problem 
\begin{equation}\label{eq:eg1fun}
f_1(x)=\sin(20|x|x),~~  x_i=-1+\frac{i-1}{1000}\in [-1,1],~~i=1,\dots,2001
\end{equation}
and $\bbP_{16}={\rm span}(1,x,\dots,x^{15})$. The Lanczos process is applied to compute the residual $\br^{(k)}$ given in \eqref{eq:LancWr2} and the threshold tolerance for cutting off $w_i^{(k)}$ is $\epsilon_w={\tu}\approx2.2 \times 10^{-16}$, the machine precision in MATLAB.  In Figure \ref{fig:Example1}, we plot the original curve of $f(x)$ in $[-1,1]$ and the computed approximation curves in the top subfigure; also the history of the dual objective values $d(\bw^{(k)})$ as well as $\|\bw^{(k)}\|_\infty$ are given in the second and the third subfigure, respectively. It then is observed that both $q=1$ and $q=2$ lead to the monotonically  increasing convergence of $d(\bw^{(k)})$; but the one with $q=2$ converges to a non-optimal minimax approximation, and  the associated sequences of $\{\|\bw^{(k)}\|_\infty\}_{k=1}^\infty$ and $\{\|\br^{(k)}\|_\infty\}_{k=1}^\infty$ behave  oscillatingly.
 \begin{figure}[thb!!!]
\hskip-12mm 
	\includegraphics[width=1.17\linewidth,height=0.65\textheight]{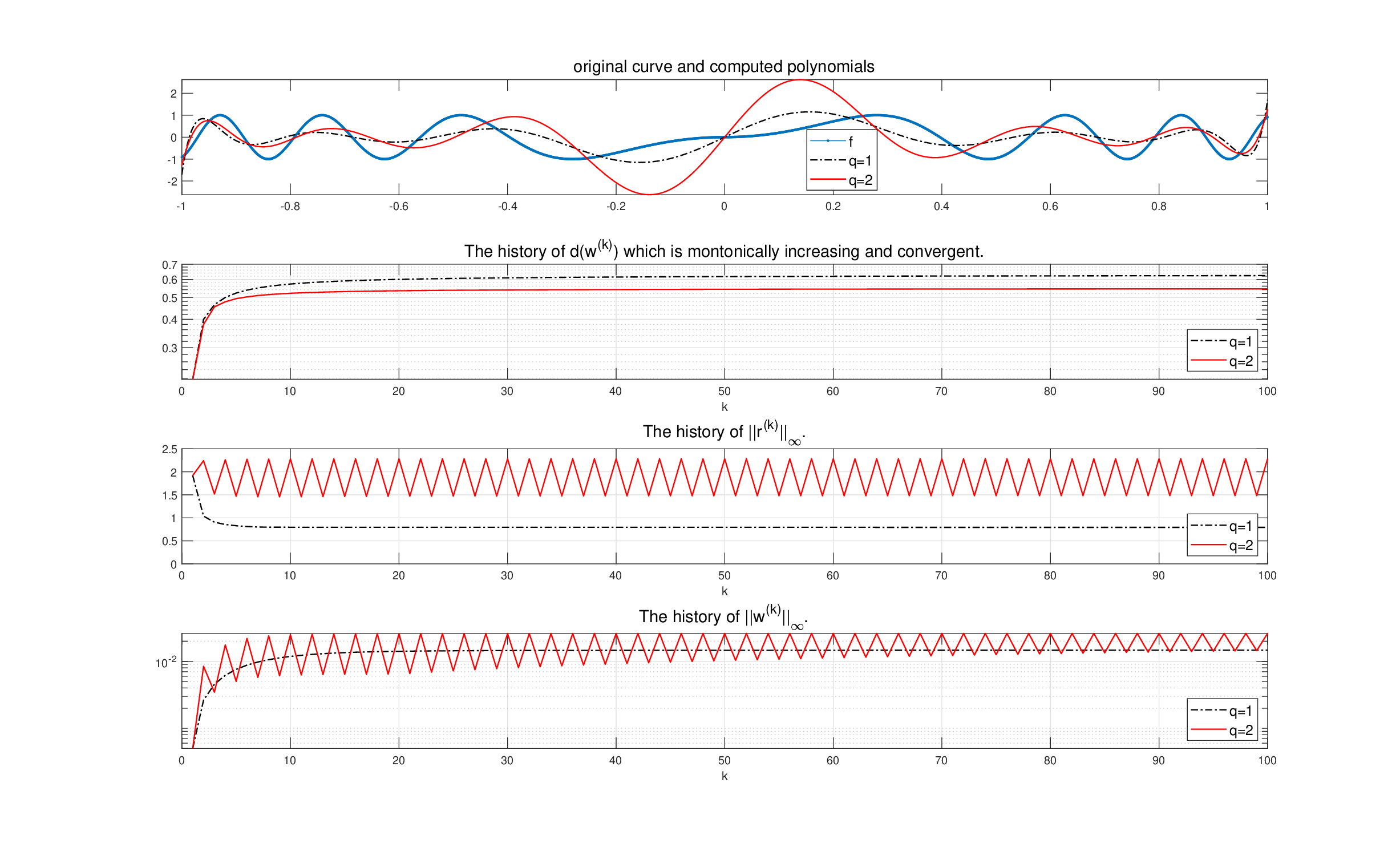}
	\vskip -3mm
\caption{\small From top to bottom: approximation polynomials,  $d(\bw^{(k)})$, $\|\br^{(k)}\|_\infty$ and $\|\bw^{(k)}\|_\infty$ from Lawson's iteration.}
\label{fig:Example1}
\end{figure}
\end{example}


\section{An interior-point method for the $L_2$-weighted dual programming}\label{sec:IP}
Recall that the linear convergence factor $\rho$ in \eqref{eq:rho}  reflects the slowest rate at which $w_i^{(k)}$ for $i\in X_3$ converges to zero. As we have pointed out that this convergence  can be slow when some nodes are sufficiently close  near some critical points in ${\cal E}$  that determine the optimal $p^*$. To fasten the convergence of $\bw^{(k)}$, we propose an interior-point method to solve the convex dual problem \eqref{eq:dualmaxdw}.

Note that for any feasible point  $\bw\in S$ of \eqref{eq:dualmaxdw}, the linear
independence constraint qualification (LICQ, see \cite[Definition 12.4]{nowr:2006}) holds. 
Following the interior-point method for the general nonlinear programming \cite[Chapter 19]{nowr:2006}, a derivation of the interior-point method is through the barrier problem of \eqref{eq:dualmaxdw}:
\begin{equation}\label{eq:dualbarrier}
\min_{\bw,\bs} -d(\bw)-\mu \sum_{i=1}^m\log s_i,~~s.t., ~\bw^{\T}\be =1,~ \bw-\bs = \mathbf{0}
\end{equation}
where $\mu>0$ is a barrier parameter and $\bs=[s_1,\dots,s_m]^{\T}$. Introduce $S=\diag(\bs)$ and variables $y\in \bbR,\bz\in \bbR^m$, and the KKT conditions for \eqref{eq:dualbarrier} read as
\begin{align}\label{eq:KKTIP} 
\left\{\begin{array}{cc}-\nabla d(\bw)-\be y-\bz&=\mathbf{0} \\S\bz-\mu \be &=\mathbf{0}\\
\bw^{\T}\be -1 &=0\\
\bw-\bs&=\mathbf{0}\end{array}\right.
\end{align}
for $(\bw(\mu),\bs(\mu),y(\mu),\bz(\mu))\in \bbR^{m}\times \bbR^{m}\times \bbR\times \bbR^{m}$, which forms the so-called central path for the primal-dual system \eqref{eq:KKTIP} with respect to the barrier parameter $\mu$. Note that the natural logarithm penalty function $\log s_i$ and the constraints $\bw-\bs=\mathbf{0}$ implicitly impose the positive constraints $\bw>0$. Indeed, by eliminating the variables $\bs$ and $\bz$, we can see that \eqref{eq:KKTIP} is equivalent to
\begin{align}\label{eq:KKTIP2} 
\left\{\begin{array}{cc}-\nabla d(\bw)-\be y-\mu W^{-1}\be&=\mathbf{0}\\
\bw^{\T}\be -1 &=0
\end{array}\right.
\end{align}
which is the KKT for the barrier problem
\begin{equation}\label{eq:dualbarrier2}
\min_{\bw^{\T}\be =1} -d(\bw)-\mu \sum_{i=1}^m\log w_i.
\end{equation}
Note that for any $\mu>0$, the objective function of \eqref{eq:dualbarrier2} is strictly convex and thus the KKT system \eqref{eq:KKTIP2} has a unique solution. Indeed, by further eliminating the second equality of \eqref{eq:KKTIP2} into the first, we have the following nonlinear system for the primal central path $\bw(\mu)$:
\begin{equation}\nonumber
-\nabla d(\bw)+\be \bw^{\T}\nabla d(\bw)+m\mu \be -\mu W^{-1}\be=\mathbf{0}.
\end{equation}

  
Computationally, we prefer the system \eqref{eq:KKTIP} as approximations of the auxiliary variables $\bs,y,\bz$ can provide information on the accuracy on how the current iterate is close to the central path. To derive the Newton step for $[\bw,\bs,y,\bz]$, we note that as long as the initial $\bw^{(0)}=\bs^{(0)}$, the Newton system of \eqref{eq:KKTIP} implies the Newton directions of $\bw$ and $\bs$ are equal; therefore, we set $\bw=\bs$ in \eqref{eq:KKTIP}  to have 
\begin{align}\label{eq:KKTIP3} 
\bk(\bw,y,\bz;\mu):=\left[ \begin{array}{cc}-\nabla d(\bw)-\be y-\bz \\W\bz-\mu \be \\
\bw^{\T}\be -1 \end{array} \right]=\mathbf{0}.
\end{align}
Denote by $\bw=\bw^{(k)}, y=y^{(k)}$ and $\bz=\bz^{(k)}$ the approximations at the $k$th iteration and let  
$[\bn_w^{\T},n_y,\bn_z^{\T}]^{\T}\in \bbR^{2m+1}$ be the Newton direction of \eqref{eq:KKTIP3}, and we have 
\begin{equation}\label{eq:Newtondir}
\left[ \begin{array}{ccc}-\nabla^2 d(\bw)&\be&I_m \\\be^{\T}&0&0  \\
I_m&0&-\Sigma^{-1} \end{array} \right]\left[\begin{array}{c}\bn_w\\-n_y \\ -\bn_z\end{array}\right]= -\left[\begin{array}{c}-\nabla d(\bw)-\be y-\bz \\\bw^{\T}\be -1 \\\bw-\mu Z^{-1}\be\end{array}\right]
\end{equation}
where $\Sigma=W^{-1}Z$ and $Z=\diag(\bz)$. With $[\bn_w^{\T},n_y,\bn_z^{\T}]^{\T}$, the Newton iteration updates $\bw,y$ and $\bz$ by
\begin{equation}\label{eq:Newtonupd}
\bw^{(k+1)}=\bw +\alpha_w^{\max}\bn_w, ~~y^{(k+1)}=y+\alpha_z^{\max}n_y,~~\bz^{(k+1)}=\bz +\alpha_z^{\max}\bn_z,
\end{equation}
where 
\begin{align*}\label{eq:alphawz}
\alpha_w^{\max}&=\max\{\alpha\in (0,1]: \bw+\alpha \bn_w\ge (1-\tau)\bw\},\\
\alpha_z^{\max}&=\max\{\alpha\in (0,1]: \bz+\alpha \bn_z\ge (1-\tau)\bz\}
\end{align*}
with $\tau\in (0,1)$, a parameter that prevents $\bw$ and $\bz$ from approaching their bounds too quickly.

Solving the system \eqref{eq:Newtondir} can be simplified. First, by eliminating $\bn_z$ as 
\begin{equation}\label{eq:nz}
\bn_z=-\bz+\mu W^{-1}\be -\Sigma \bn_w,
\end{equation}
we have 
\begin{equation}\nonumber
\left[ \begin{array}{cc}-\nabla^2 d(\bw)+\Sigma&\be\\\be^{\T}&0\end{array} \right]\left[\begin{array}{c}\bn_w\\-n_y \end{array}\right]= -\left[\begin{array}{c}\bh_1(\bw,y,\bz) \\\bh_2(\bw,y,\bz)\end{array}\right]
\end{equation}
where 
$$\bh_1=\bh_1(\bw,y,\bz):=-\nabla d(\bw)-\be y -\mu W^{-1}\be,~~\bh_2=\bh_2(\bw,y,\bz):=\bw^{\T}\be -1.$$
Consequently, 
\begin{align}\label{eq:Newtondir3w} 
\bn_w&=(-\nabla^2 d(\bw)+\Sigma)^{-1}(n_y \be -\bh_1),\\\label{eq:Newtondir3y} 
n_y&=\frac{\be^{\T}(-\nabla^2 d(\bw)+\Sigma)^{-1} \bh_1-\bh_2}{\be^{\T}(-\nabla^2 d(\bw)+\Sigma)^{-1} \be}.
\end{align}
Note that there are two linear systems involved for computing the Newton direction, namely, 
$$\bv_e=(-\nabla^2 d(\bw)+\Sigma)^{-1}\be, ~~\mbox{and}~~\bv_h=(-\nabla^2 d(\bw)+\Sigma)^{-1}\bh_1.$$
Although both are related with an $m$-by-$m$ coefficient matrix $(-\nabla^2 d(\bw)+\Sigma)^{-1}$, we will show that by relying on the Hessian \eqref{eq:hess_dw_orth}, $\bv_e$ and $\bv_h$ can be computed with flops $O(n^3)$ using  the Sherman-Morrison-Woodbury formula (see e.g., \cite[Equ. (2.1.4)]{govl:2013}). 
To this end, let $W^{\frac12}\Psi=\Psi_WR_W ~(\Psi_W\in \bbC^{m\times n},~R_W\in \bbC^{n\times n})$ be the thin QR decomposition of $W^{\frac12}\Psi$, and write $$\Psi_W^{\HH}R=K^{\tt r}+{\texttt{i}}K^{\tt i}\in \bbC^{n\times m},~~R=\diag{(\br)},~~ K^{\tt r}={\rm Re}(\Psi_W^{\HH}R),~ K^{\tt i}={\rm Im}(\Psi_W^{\HH}R)\in \bbR^{n\times m}$$
to have 
\begin{equation}\label{eq:mediHess}
{\rm Re}(R^{\HH}\Psi_W \Psi_W^{\HH}R)=(K^{\tt r})^{\T}K^{\tt r}+(K^{\tt i})^{\T}K^{\tt i}=K^{\T}K,~~K=\left[\begin{array}{c} K^{\tt r}\\K^{\tt i}\end{array}\right]\in \bbR^{2n\times m}.
\end{equation}
Then from \eqref{eq:hess_dw_orth} and $\Sigma=W^{-1}Z$, it holds 
\begin{align}\nonumber
(-\nabla^2 d(\bw)+\Sigma)^{-1}&=\left(2W^{-\frac12} {\rm Re}(R^{\HH}\Psi_W \Psi_W^{\HH}R)W^{-\frac12}+\Sigma\right)^{-1}
\\\nonumber
&=\left(2W^{-\frac12} K^{\T}KW^{-\frac12}+\Sigma\right)^{-1}\\
&=\Sigma^{-1}-2(\Sigma Z)^{-\frac12}K^{\T}\left(I+2KZ^{-1}K^{\T}\right)^{-1}K(\Sigma Z)^{-\frac12};\label{eq:SMW}
\end{align}
thus computing $\bv_e$ and $\bv_h$ only need to solve two linear systems associated with a symmetric and positive definite  coefficient matrix $\left(I+2KZ^{-1}K^{\T}\right)^{-1}\in \bbR^{2n\times 2n}$.  For the real case,   it is worth pointing out that $K$ in \eqref{eq:mediHess} is an $n$-by-$m$ matrix and thus  $\left(I+2KZ^{-1}K^{\T}\right)^{-1}\in \bbR^{n\times n}$.
Furthermore, in the case of  $\bbP_n={\rm span}(1,x,\cdots,x^{n-1})$,  the Hessian \eqref{eq:hess_dw_orth} can be formed by the Arnoldi  process in $O(mn)$ flops (refer to our discussion in section \ref{subsec_revistLawson}), and hence,  computing a Newton step of \eqref{eq:Newtonupd} can be done in $O(mn)+O(n^3)$ flops.

For the barrier parameter $\mu^{(k)}$, we employ an adaptive strategy described in \cite[Section 19.3]{nowr:2006} in which 
\begin{equation}\label{eq:muk}
\mu^{(k+1)} =\sigma^{(k)}\frac{(\bw^{(k)})^{\T}\bz^{(k)}}{m}
\end{equation} 
where
$$
\sigma^{(k)}=0.1\min\left(\frac{1-\xi^{(k)}}{20\xi^{(k)}},2\right)^3 \mbox{ with }~\xi^{(k)}=\frac{\min_i\bw_i^{(k)}\bz_i^{(k)}}{(\bw^{(k)})^{\T}\bz^{(k)}/m}.
$$
With these techniques, we can now present the algorithm framework of the interior-point method in Algorithm \ref{alg:IP}. We refer to \cite[Theorem 19.1]{nowr:2006} for the convergence and \cite[Section 19.8]{nowr:2006} for the superlinear convergence rate of the interior-point method for the general nonlinear programming.

\begin{algorithm}[t]
\caption{An interior-point method for \eqref{eq:bestf}} \label{alg:IP}
\begin{algorithmic}[1]
\renewcommand{\algorithmicrequire}{\textbf{Input:}}
\renewcommand{\algorithmicensure}{\textbf{Output:}}
\REQUIRE $\{(x_j,f_j)\}_{j=1}^m$ and $0\le n+1\le m$ and a function basis $\{\psi_i\}_{i=1}^n$ for $\bbP_n$
\ENSURE  the (approximation) solution $p^*\in \bbP_n$ of \eqref{eq:bestf} 
        \smallskip

\STATE  Let $k=0$ and chose $\tau\in(0,1)$,   $0<\bw^{(0)}\in S$ and $\bz^{(0)}>0$ and a  threshold tolerance $\epsilon_w$ for the weights;
\STATE (optional) Remove nodes $x_i$ with $\bw_i^{(k)}<\epsilon_w$;
\STATE Solve \eqref{eq:q-dual_q2} with the weight vector $\bw^{(k)}$ to obtain  $d(\bw^{(k)})$,  $\nabla d(\bw^{(k)})$ \eqref{eq:grad_d} and $\nabla^2 d(\bw^{(k)})$ \eqref{eq:hess_dw_orth};
\STATE  Compute $\bn_w,n_y$ of \eqref{eq:Newtondir3w}  and \eqref{eq:Newtondir3y}, respectively, by using \eqref{eq:SMW}, and then compute $\bn_z$ by \eqref{eq:nz};
\STATE Update $\bw^{(k+1)},y^{(k+1)}$ and $\bz^{(k+1)}$ according to \eqref{eq:Newtonupd}, and update $\mu^{(k+1)}$ by \eqref{eq:muk};
 
\STATE Stop if the given criterion is met and the approximation $p^*$ is obtained from \eqref{eq:q-dual_q2} with the weight $\bw^{(k+1)}$; otherwise $k=k+1$ and goto step 2.
\end{algorithmic}
\end{algorithm}
We have a few remarks for Algorithm \ref{alg:IP}.
\begin{remark}\label{rmk:IP}
\begin{itemize}
\item[(1)] The nature of the logarithm barrier  in \eqref{eq:dualbarrier} can prevent ``accidentally'' 
but permanently setting certain critical weights to zeros; on the other hand, as the interior-point employs the Newton direction, the fast convergence  is able to quickly deduce non-critical weights corresponding to nodes in ${\cal J}$ \eqref{eq:AJ} to zero. This is observed in our numerical experiments. In order to handle the sequential systems more effectively, we include a weight-filtering option in step 2 of Algorithm \ref{alg:IP} to remove  nodes whose corresponding weights are smaller than a sufficiently small  threshold tolerance $\epsilon_w$. Such a weight-filtering option can also be incorporated into Lawson's iteration  before Step 2 of  Algorithm \ref{alg:lawson}, and our numerical tests in section \ref{sec:numer} will evaluate the effectiveness of this procedure.

\item[(2)] For the stopping criterion in step 6, we choose to terminate the iteration whenever one of the following inequalities is fulfilled:
\begin{equation}\label{eq:stoprule}
\frac{\left|d(\bw^{(k+1)})-d(\bw^{(k)})\right|}{d(\bw^{(k)})}<\epsilon_d,~~\left\|\bk(\bw^{(k)},y^{(k)},\bz^{(k)};\mu^{(k)})\right\|_\infty<\epsilon_K,~~k>k_{\max}
\end{equation}
where $\left\|\bk(\bw^{(k)},y^{(k)},\bz^{(k)};\mu^{(k)})\right\|_\infty$ is the error of the perturbed KKT given in \eqref{eq:KKTIP3}.
\item[(3)] It is known that the interior-point method is an efficient method for the equivalent linear programming \eqref{eq:L1LP-b} for the real case, but \eqref{eq:L1LP-b} is inapplicable for the complex case. Even for the real case,  \eqref{eq:L1LP-b} doubles the constraints (note that the slack  variable $\bs$ in \eqref{eq:L1LP-bs} is of size $2m$); also, the interior-point method for \eqref{eq:dualmaxdw} can filter out non-critical nodes and dynamically determine the equioscillation points  in Theorem \ref{thm:Alternation} during the course of iterations. We will present numerical examples in section \ref{sec:numer}.
\end{itemize}
\end{remark}
\section{Numerical experiments}\label{sec:numer}
We implement Algorithm \ref{alg:IP} in 
MATLAB R2022a and carry out numerical tests on a DELL Inspiron 7700 AIO with 16 GB 2666 MHz DDR4 and the unit machine roundoff $\tu=2^{-52}\approx 2.2\times 10^{-16}$. 
Particular parameters for the stopping rule \eqref{eq:stoprule} are set as $\epsilon_d = \epsilon_K=10^{-10}$. Also, the initial barrier parameter $\mu^{(0)}=10^{-5}$ and the initial weight vector is  $\bw^{(0)}=\be/m$.

\subsection{Numerical evaluations on real cases}\label{subsec:real}
\begin{example}\label{eg:2}
We first present the numerical behavior of the interior-point method (Algorithm \ref{alg:IP}) by comparing the classical Lawson's iteration for the function \eqref{eq:eg1fun} in Example \ref{eg:1}. In order to evaluate how fast the weights converge, we inactivate the filtering procedure for the weights by setting the threshold tolerance $\epsilon_w=0$ for both the interior-point method and Lawson's iteration. The interior-point method  (Algorithm \ref{alg:IP}) stops at $k=18$, while Lawson's iteration does not meet the stopping rule with $\epsilon=10^{-10}$ in step 4 of Algorithm \ref{alg:lawson} for $k=100$. In Figure \ref{fig:Example2}, we plot the approximation polynomials, the sequences of  $\{d(\bw^{(k)})\}$ and  $\{\|\bw^{(k)}\|_\infty\}$. It can be seen that the interior-point method converges fast; more interestingly, we notice that the largest weight of $w_i^{(k)}$ from the interior-point method quickly reaches the limit, but the one from Lawson's iteration converges very slowly; as $\be^{\T}\bw^{(k)}=1$, the slow convergence of $\{\|\bw^{(k)}\|_\infty\}$ implies that no node dominants significantly, and therefore, the equioscillation points  in Theorem \ref{thm:Alternation} that determine the best approximation $p^*$ do not emerge quickly. To demonstrate this fact, we plot the weights from the interior-point method and Lawson's iteration for various maximal numbers of iterations in Figure \ref{fig:Example3}. One can easily see that $18$ equioscillation points associated with $\bbP_{16}$ emerge in the interior-point method, but the weights from Lawson's iterations scatter. Notice that near each equioscillation point, there are many nodes whose weights decay continuously but slowly, and thus, it is hard to distinguish the $18$ equioscillation points in Lawson's iteration. To extend the test, we apply the two methods for finding the best approximation $p^*$ for the function \eqref{eq:eg1fun} in Example \ref{eg:1} on $\bbP_{21}$. It turns out that the interior-point method stops at $k=20$ and finds the exactly $22$ equioscillation points. The numerical behaviors are given in Figure \ref{fig:Example3b}.
\begin{figure}[h!!!]
\hskip-12mm 
	\includegraphics[width=1.18\linewidth,height=0.6\textheight]{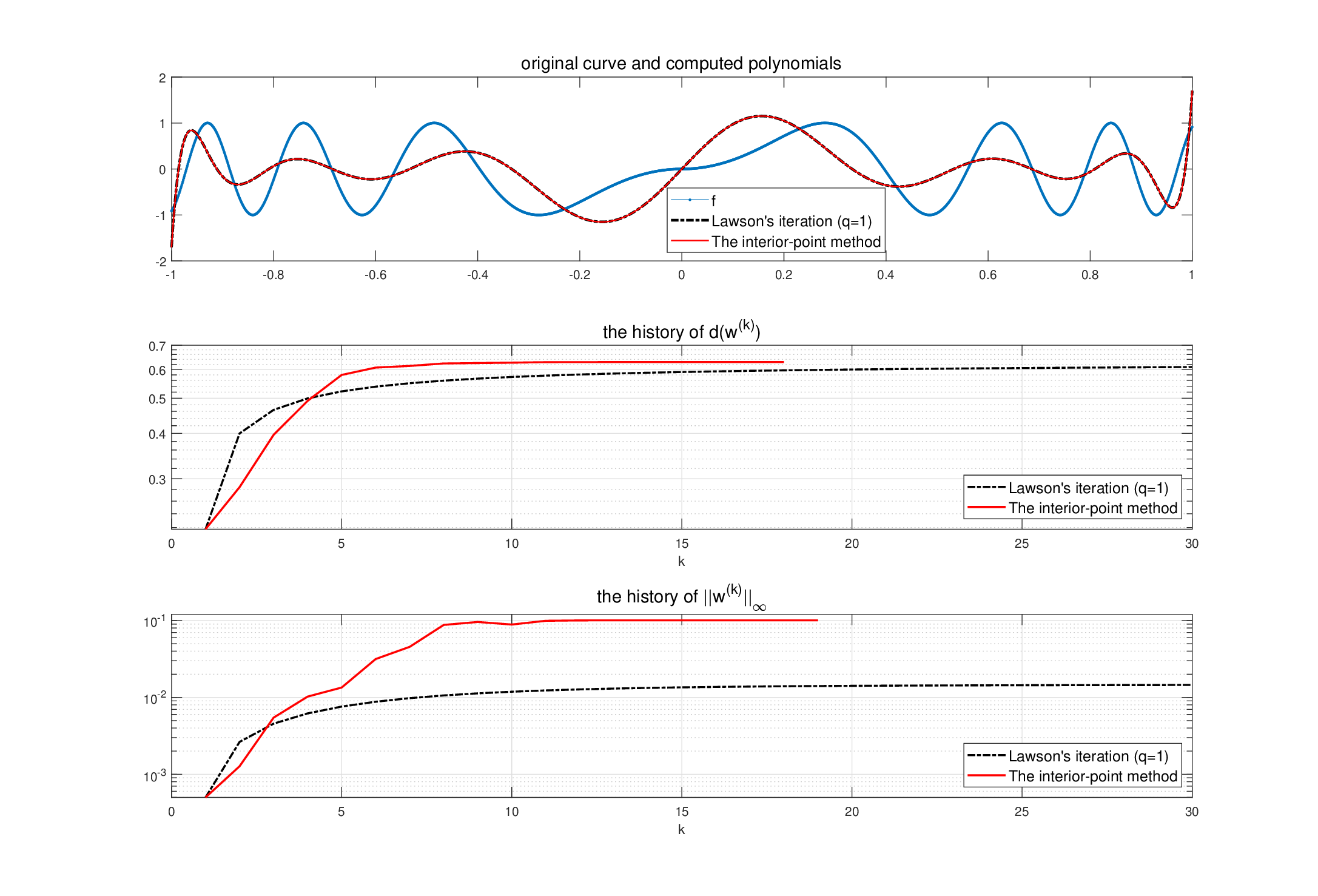}
	\vspace{-1.5cm}
\caption{\small From top to bottom: approximation polynomials,  $d(\bw^{(k)})$  and $\|\bw^{(k)}\|_\infty$ from Lawson's iteration and the interior-point method (Algorithm \ref{alg:IP}).}
\label{fig:Example2}
\end{figure}

 \begin{figure}[h!!!]
\hskip-13mm
	\includegraphics[width=1.18\linewidth,height=0.48\textheight]{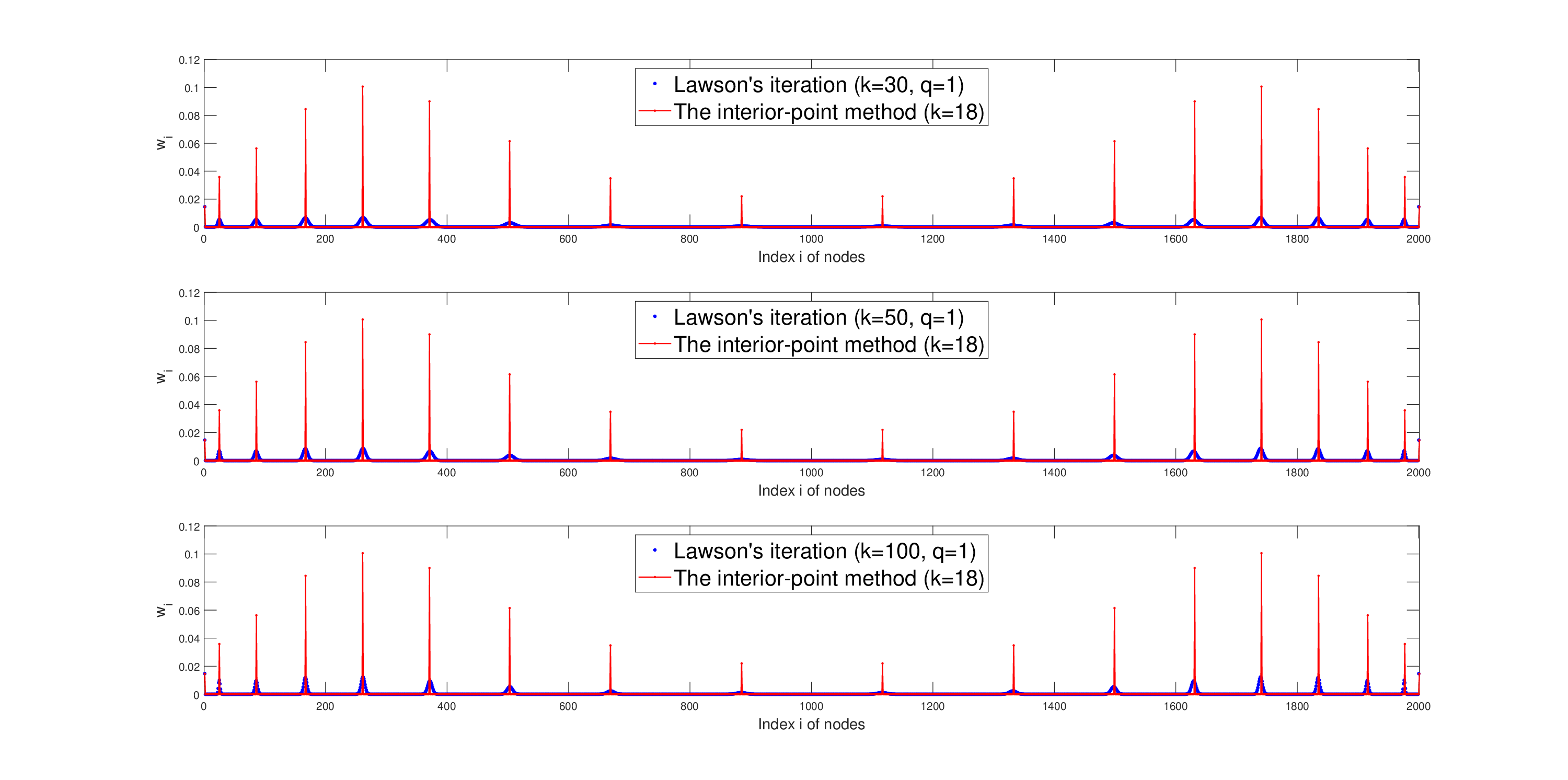}
	\vspace{-1.5cm}
	\caption{Weights $w_i$ for $1\le i\le m$ from Lawson's iteration (top: $k=30$, middle: $k=50$, bottom: $k=100$) and the interior-point method ($k=18$) for Example \ref{eg:1} on $\bbP_{16}$.}
\label{fig:Example3}
\end{figure}

 \begin{figure}[h!!!]
\hskip-12mm
	\includegraphics[width=1.18\linewidth,height=0.48\textheight]{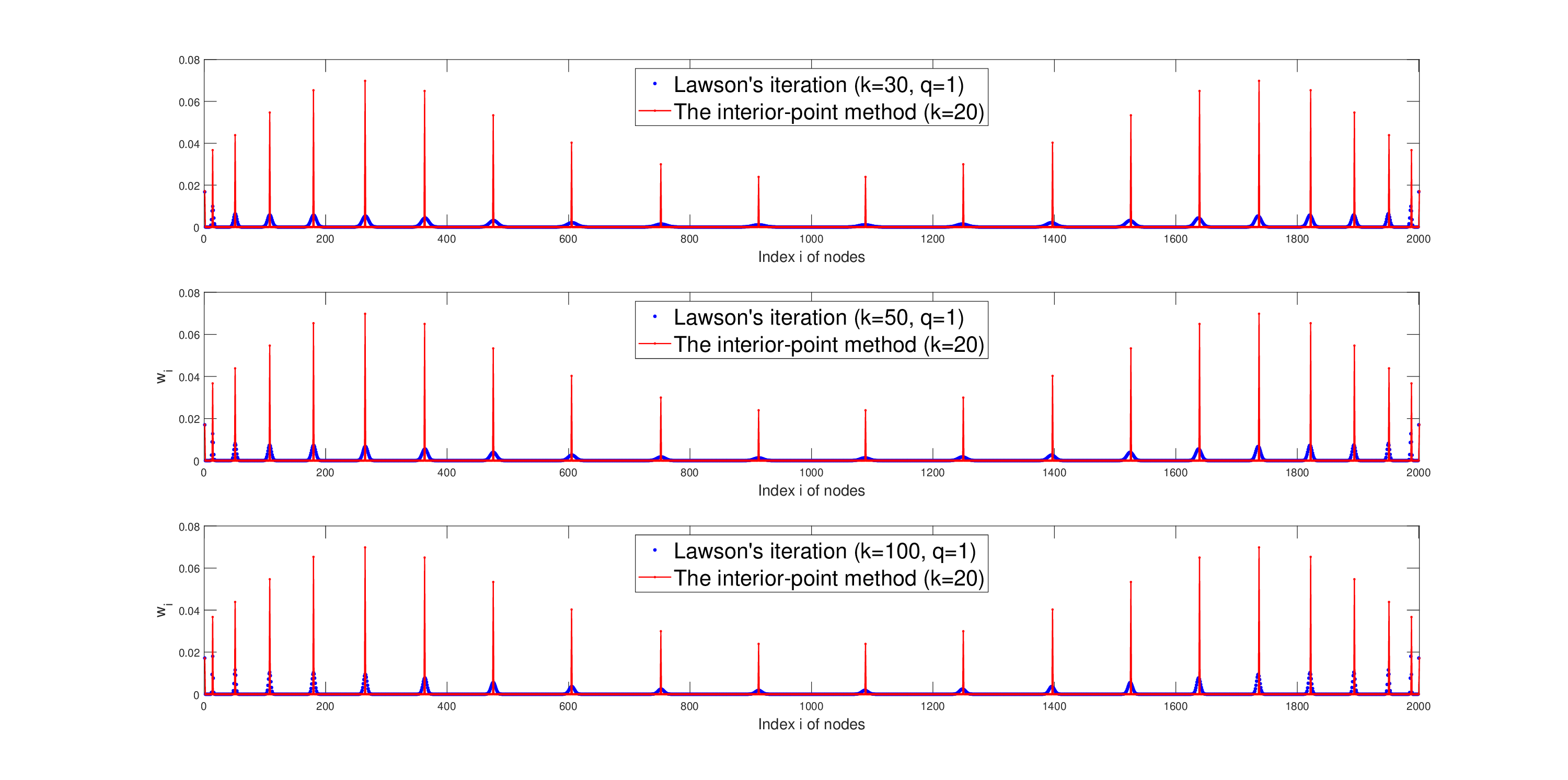}
	\caption{Weights $w_i$ for $1\le i\le m$ from Lawson's iteration (top: $k=30$, middle: $k=50$, bottom: $k=100$) and the interior-point method ($k=20$) for Example \ref{eg:1} on $\bbP_{21}$.}
\label{fig:Example3b}
\end{figure}
\end{example}

\begin{example}\label{eg:3}
In this example, we shall activate the filtering procedure (i.e., step 2 in Algorithm \ref{alg:IP}) for the weights by setting various values for the threshold tolerance $\epsilon_w$ for both the interior-point method and Lawson's iteration.  Besides the test function $f_1(x)$ given in \eqref{eq:eg1fun}, we provide numerical results for the well-known Runge function 
\begin{equation}\label{eq:eg2fun}
f_2(x)=\frac{1}{1+25x^2},~~  x_i=-1+\frac{i-1}{1000}\in [-1,1],~~i=1,\dots,2001.
\end{equation}
With different levels of threshold tolerance $\epsilon_w$ on $\bbP_{21}$ and $\bbP_{31}$, in each step, we remove the nodes for which the associated weights are less than $\epsilon_w$. For both methods, we report on the number $k$ of iterations, the final maximal modulus error $\Vert \br^{(k)} \Vert_{\infty}$, the final dual objective function value $d(\bw^{(k)})$, the number of final nodes and the consumed CPU (in second) in Table \ref{tab:eg3}. Again, it is observed that the interior-point method also performs well with this weight-filtering  procedure.
\begin{table}[H]
\newcommand{\tabincell}[2]{\begin{tabular}{@{}#1@{}}#2\end{tabular}}
\renewcommand{\arraystretch}{1}
\caption{Numerical results in Example \ref{eg:2} and Example \ref{eg:3}}
\begin{center}
{\begin{scriptsize} \tabcolsep0.08in
\begin{tabular}{|c||c|c|c|c|c|c|}
      \hline
      &\multicolumn{6}{|c|}{\tabincell{c}{Test on $f_1$ \eqref{eq:eg1fun}  on $\bbP_{21}$ and $m=2001$}}\\
      \hline
       $\epsilon_w$          & {\rm  Methods} & {\rm \# of iterations $k$} & {\tt $\Vert \br^{(k)} \Vert_{\infty}$} & {\tt $d(\bw^{(k)})$} & {\rm \# of final nodes} & {\rm CPU (s) } \\
      \hline
      \bf ${10^{-6}}/{m}$    & \tabincell{c}{Lawson's iteration \\ Interior-point method} & \tabincell{c}{1000\\22} & \tabincell{c}{3.4238e-1\\3.4235e-1} & \tabincell{c}{1.1710e-1\\1.1720e-1} & \tabincell{c}{232\\ {\bf 22}} & \tabincell{c}{0.38\\0.05} \\
      \hline
      \bf ${10^{-5}}/{m}$    & \tabincell{c}{Lawson's iteration \\ Interior-point method} & \tabincell{c}{1000\\21} & \tabincell{c}{3.4238e-1\\3.4235e-1} & \tabincell{c}{1.1710e-1\\1.1720e-1} & \tabincell{c}{218\\{\bf 22}} & \tabincell{c}{0.34\\0.05}\\
      \hline
      \bf ${10^{-4}}/{m}$    & \tabincell{c}{Lawson's iteration \\ Interior-point method} & \tabincell{c}{1000\\20} & \tabincell{c}{3.4238e-1\\3.4235e-1} & \tabincell{c}{1.1710e-1\\1.1720e-1} & \tabincell{c}{206\\{\bf 22}} & \tabincell{c}{0.27\\0.09}\\
      \hline
      &\multicolumn{6}{|c|}{\tabincell{c}{Test on $f_1$ \eqref{eq:eg1fun}   on $\bbP_{31}$ and $m=2001$}}\\
      \hline
      $\epsilon_w$           & {\rm  Methods} & {\rm \# of iterations $k$} &  {\tt $\Vert \br^{(k)} \Vert_{\infty}$} & {\tt $d(\bw^{(k)})$}  & {\rm \# of final nodes} & {\rm CPU (s) } \\
      \hline
      \bf ${10^{-6}}/{m}$    & \tabincell{c}{Lawson's iteration \\ Interior-point method} & \tabincell{c}{1000\\26} & \tabincell{c}{7.6031e-3\\7.6028e-3} &
      \tabincell{c}{ 5.7750e-5\\5.7802e-5} & \tabincell{c}{234\\{\bf 32}} & \tabincell{c}{0.63\\0.09}\\
      \hline
      \bf ${10^{-5}}/{m}$    & \tabincell{c}{Lawson's iteration \\ Interior-point method} & \tabincell{c}{1000\\26} & \tabincell{c}{7.6031e-3\\7.6028e-3} &
      \tabincell{c}{ 5.7750e-5\\5.7802e-5} & \tabincell{c}{222\\ {\bf 32}} & \tabincell{c}{0.48\\0.08}\\
      \hline
      \bf ${10^{-4}}/{m}$    & \tabincell{c}{Lawson's iteration \\ Interior-point method} & \tabincell{c}{1000\\25} & \tabincell{c}{7.6031e-3\\7.6028e-3} &
      \tabincell{c}{ 5.7750e-5\\5.7802e-5} & \tabincell{c}{206\\{\bf 32}} & \tabincell{c}{0.63\\0.02}\\
      \hline 
      \hline
      & \multicolumn{6}{|c|}{\tabincell{c}{Test on $f_2$ \eqref{eq:eg2fun} on $\bbP_{21}$ and $m=2001$}}\\
      \hline
      $\epsilon_w$           & {\rm  Methods} & {\rm \# of iterations $k$} & {\tt $\Vert \br^{(k)} \Vert_{\infty}$} & {\tt $d(\bw^{(k)})$} & {\rm \# of final nodes} & {\rm CPU (s) } \\
      \hline
      \bf ${10^{-6}}/{m}$    & \tabincell{c}{Lawson's iteration \\ Interior-point method} & \tabincell{c}{1000\\21} & \tabincell{c}{2.8474e-1\\2.8473e-1} & \tabincell{c}{8.1000e-2\\ 8.1074e-2} & \tabincell{c}{236\\{\bf 22}} & \tabincell{c}{0.31\\0.02}\\
      \hline
      \bf ${10^{-5}}/{m}$    & \tabincell{c}{Lawson's iteration \\ Interior-point method} & \tabincell{c}{1000\\21} & \tabincell{c}{2.8474e-1\\2.8473e-1} & \tabincell{c}{8.1000e-2\\ 8.1074e-2} & \tabincell{c}{222\\{\bf 22}} & \tabincell{c}{0.31\\0.03}\\
      \hline
      \bf ${10^{-4}}/{m}$    & \tabincell{c}{Lawson's iteration \\ Interior-point method} & \tabincell{c}{1000\\20} & \tabincell{c}{2.8474e-1\\2.8473e-1} & \tabincell{c}{8.1000e-2\\ 8.1074e-2} & \tabincell{c}{204\\{\bf 22}} & \tabincell{c}{0.25\\0.05}\\
      \hline
      & \multicolumn{6}{|c|}{ \tabincell{c}{Test on $f_2$ \eqref{eq:eg2fun}  on $\bbP_{31}$ and $m=2001$}}\\
      \hline
      $\epsilon_w$           & {\rm  Methods} & {\rm \# of iterations $k$} & {\tt $\Vert \br^{(k)} \Vert_{\infty}$} & {\tt $d(\bw^{(k)})$} & {\rm \# of final nodes} & {\rm CPU (s) } \\
      \hline
      \bf ${10^{-6}}/{m}$    & \tabincell{c}{Lawson's iteration \\ Interior-point method} & \tabincell{c}{1000\\23} & \tabincell{c}{4.1257e-2\\4.1255e-2} & \tabincell{c}{1.7005e-3\\ 1.7020e-3} & \tabincell{c}{236\\{\bf 32}} & \tabincell{c}{0.53\\0.02}\\
      \hline
      \bf ${10^{-5}}/{m}$    & \tabincell{c}{Lawson's iteration \\ Interior-point method} & \tabincell{c}{1000\\22} & \tabincell{c}{4.1257e-2\\4.1255e-2} & \tabincell{c}{1.7005e-3\\ 1.7020e-3} & \tabincell{c}{224\\{\bf 32}} & \tabincell{c}{0.58\\0.05}\\
      \hline
      \bf ${10^{-4}}/{m}$    & \tabincell{c}{Lawson's iteration \\ Interior-point method} & \tabincell{c}{1000\\21} & \tabincell{c}{4.1257e-2\\4.1255e-2} & \tabincell{c}{1.7005e-3\\ 1.7020e-3 }& \tabincell{c}{200\\{\bf 32}} & \tabincell{c}{0.27\\0.03}\\
      \hline
     \end{tabular}
\end{scriptsize}
}\end{center}
\label{tab:eg3}
\end{table}
\end{example}

\subsection{Numerical evaluations on complex cases}\label{subsec:complex}
We now test the interior-point method for some complex cases. 
\begin{example}\label{eg:4}
We first choose  the following complex-valued  function
\begin{equation}\label{eq:eg1Cfun}
	g_1(z) = (2z+1)^{-1/2},~ z_k=e^{-\frac{\pi}{2}{\tt i}+\frac{(k-1)\pi {\tt i}}{2000}}, k =1,\dots,2001
\end{equation}
with $m=2001$ uniform nodes on the semi-disc and   $\bbP_9=\{1,z,\dots,z^8\}$. 

We first inactivate the filtering procedure with the threshold tolerance $\epsilon_w=0$. The interior-point method stops at $k=24$ in $0.09$s while Lawson's iteration does not meet the stopping rule  within the maximal iteration number $k=100$ (see the middle subfigure of Figure \ref{fig:Example4a}). 
At the top of Figure \ref{fig:Example4}, we plot the maximum modulus error circle (the red circle) and the error curve $g_1(z)-p^*(z)$ at samples (the blue points). It is observed that the error curve touches the red circle 10 times; also it can be seen that the radius of interior-point method is smaller than Lawson's iteration for all chosen $\epsilon_w$ (see Table \ref{tab:eg45}). We also plot  the approximation polynomial at the bottom of Figure  \ref{fig:Example4}.
The history of $\{d(\bw^{(k)})\},\{\Vert \br^{(k)}\Vert_{\infty}\}$ and the final weights $\{w_i\}_{i=1}^{2001}$ are shown in Figure \ref{fig:Example4a}. One can see from the top of Figure \ref{fig:Example4a} that the interior-point method converges much faster and the computed dual objective function is  larger than that of the Lawson's iteration.  Most interestingly,  the bottom of Figure \ref{fig:Example4a} indicates that the interior-point method is capable of finding the exactly $10$ reference points that support the best polynomial; this is difficult for Lawson's iteration (see Figure \ref{fig:Example4a} and Table \ref{tab:eg45}). More detailed information on the iteration number $k$ (with $k_{\max}=1000$),  the final error $\Vert \br^{(k)} \Vert_{\infty}$, the final dual objective function value $d(\bw^{(k)})$, the number of remained nodes and the consumed CPU (in second) are reported in Table \ref{tab:eg45} with various levels of threshold $\epsilon_w$ on $\bbP_{9}$ and $\bbP_{16}$.

\begin{figure}[tbp!!!]
	\centering
	\includegraphics[width=0.9\linewidth,height=0.67\textheight]{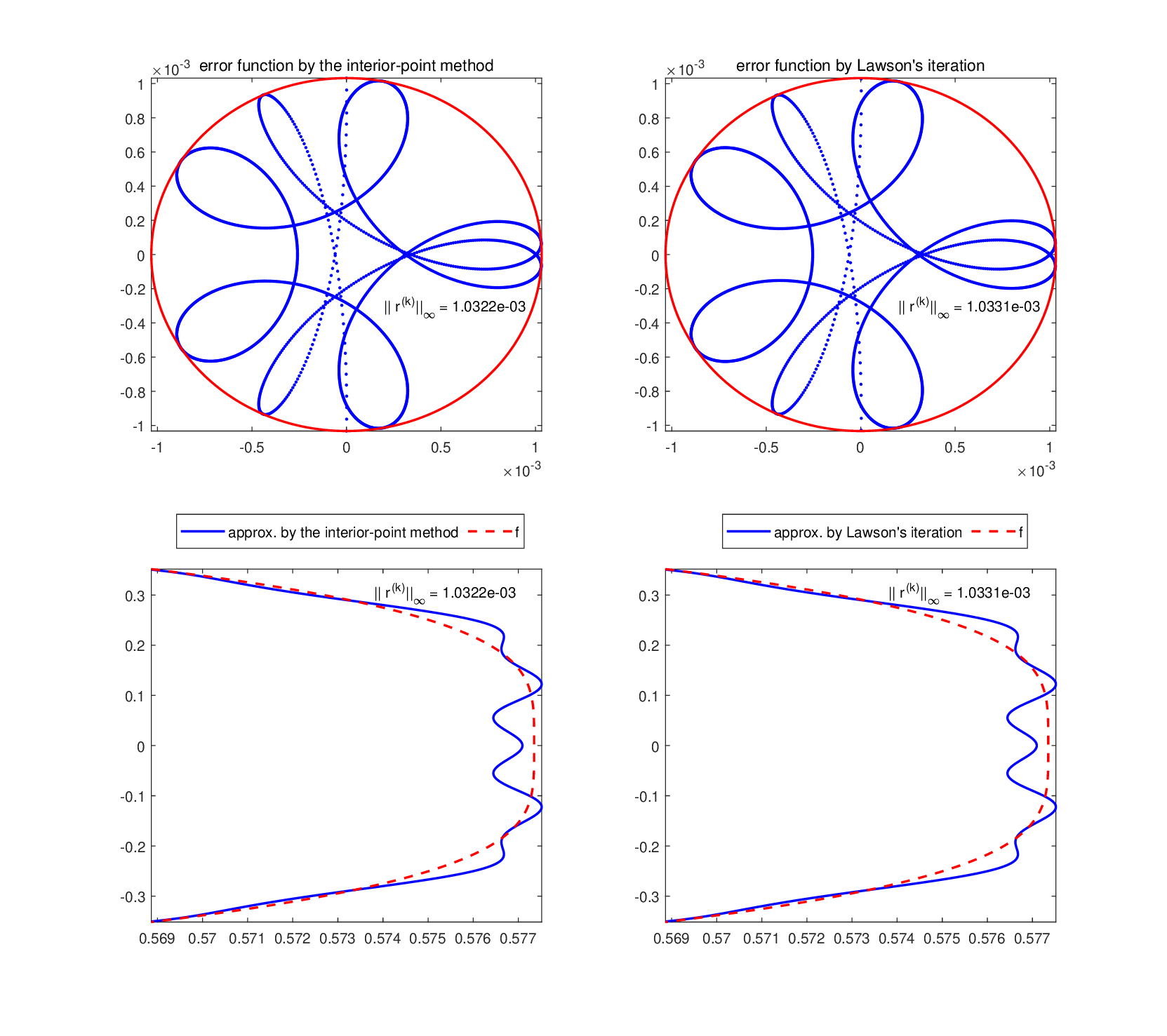}
	\vspace{-.5cm}
	\caption{\small From top to bottom: the error curve and the approximation function by the interior-point method (left) and Lawson's iteration (right) for Example \ref{eg:4} on $\bbP_9$.}
	\label{fig:Example4}
\end{figure}
\vspace{-.5cm}

\begin{figure}[H]
	\hskip-13mm
	\includegraphics[width=1.17\linewidth,height=0.59\textheight]{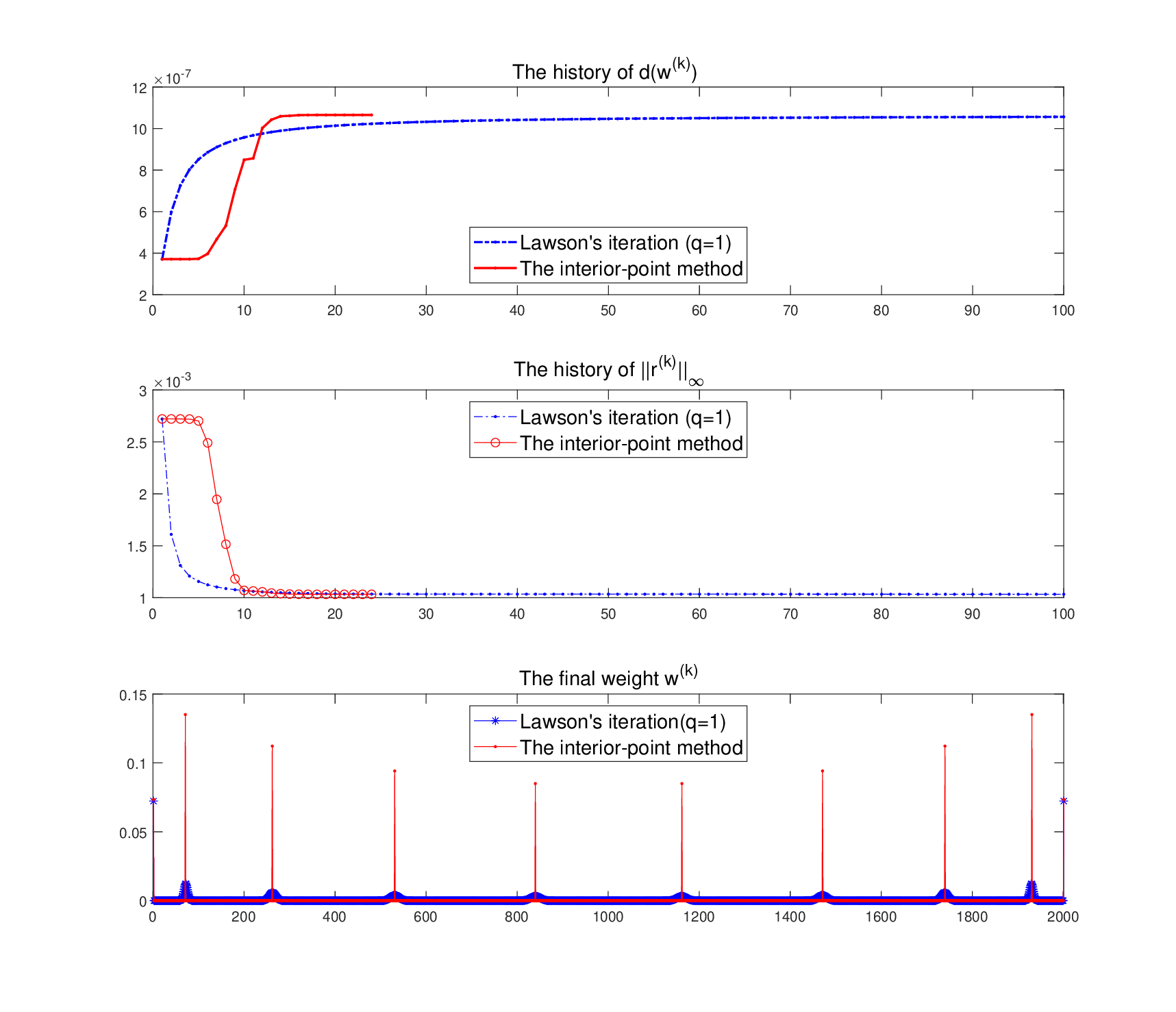}
	\caption{\small From top to bottom: $\{d(\bw^{(k)})\},\{\Vert \br^{(k)}\Vert_{\infty}\}$ and final weights $\{w_i\}_{i=1}^{2001}$ from Lawson's iteration (blue) and interior-point method (red) for Example \ref{eg:4} on $\bbP_{9}$.}
	\label{fig:Example4a}
\end{figure}

\end{example}

\begin{example}\label{eg:5}
	This complex example is from \cite[Fig 4.3, Sec.4]{natr:2020} 
	\begin{equation}\label{eq:eg2Cfun}
		g_2(z)=(1+z^4)^{1/2},~z_k=e^{{\tt i}\pi/4\tanh\left(-12+\frac{24(k-1)}{1000}\right)},~k=1,\dots,2001.
	\end{equation}
	The curve of the  error function and the approximation best polynomial on $\bbP_{21}$ with the maximal iteration number $100$ are shown in Figure \ref{fig:Example5}. The history of $\{d(\bw^{(k)})\},\{\Vert \br^{(k)}\Vert_{\infty}\}$ and the final weights $\{w_i\}_{i=1}^{2001}$ can be seen in Figure \ref{fig:Example5a}. It is clearly that the interior-point method also performs well for this example, and is effective in finding the reference points.

\begin{table}[H]
		\newcommand{\tabincell}[2]{\begin{tabular}{@{}#1@{}}#2\end{tabular}}
		\renewcommand{\arraystretch}{1}
		\caption{Numerical results in Example \ref{eg:4} and Example \ref{eg:5}}
		\begin{center}\label{tab:eg45}
			{\begin{scriptsize} \tabcolsep0.08in
					\begin{tabular}{|c||c|c|c|c|c|c|}
						\hline
						&\multicolumn{6}{|c|}{ \tabincell{c}{Test on $g_1$ \eqref{eq:eg1Cfun}  on $\bbP_{9}$ and $m=2001$}}\\
						\hline
						$\epsilon_w$&   {\rm  Methods} &  {\rm \# of iterations $k$} & {\tt $\Vert \br^{(k)}\Vert_{\infty}$} & {\tt $d(\bw^{(k)})$} & {\rm \# of final nodes} & {\rm CPU (s) } \\
						\hline 
						\bf ${10^{-6}}/{m}$    &\tabincell{c}{Lawson's iteration \\Interior-point method} &\tabincell{c}{1000\\27} & \tabincell{c}{1.0323e-3\\1.0322e-3} &\tabincell{c}{1.0646e-6\\1.0654e-6} & \tabincell{c}{242 \\ {\bf 10}} & \tabincell{c}{ 0.73\\ 0.06} \\
						\hline
						\bf ${10^{-5}}/{m}$    &\tabincell{c}{Lawson's iteration \\Interior-point method}&\tabincell{c}{1000\\27} & \tabincell{c}{1.0323e-3\\1.0322e-3} &\tabincell{c}{1.0646e-6\\1.0654e-6} & \tabincell{c}{224 \\ {\bf 10}} & \tabincell{c}{ 0.58\\ 0.03} \\
						\hline
						\bf ${10^{-4}}/{m}$    &\tabincell{c}{Lawson's iteration \\Interior-point method}&\tabincell{c}{1000\\29} & \tabincell{c}{1.0323e-3\\1.0322e-3} &\tabincell{c}{1.0646e-6\\1.0654e-6} & \tabincell{c}{206\\ {\bf 10}} & \tabincell{c}{ 0.55\\ 0.08} \\
						\hline
						&\multicolumn{6}{|c|}{ \tabincell{c}{Test on $g_1$ \eqref{eq:eg1Cfun}   on $\bbP_{16}$ and $m=2001$}}\\
						\hline
						$\epsilon_w$&   {\rm  Methods} &  {\rm \# of iterations $k$} & {\tt $\Vert \br^{(k)}\Vert_{\infty}$} & {\tt $d(\bw^{(k)})$} & {\rm \# of final nodes} & {\rm CPU (s) } \\
						\hline
						\bf ${10^{-6}}/{m}$    &\tabincell{c}{Lawson's iteration \\Interior-point method} &\tabincell{c}{1000\\36} & \tabincell{c}{1.0529e-5 \\1.0528e-5} &\tabincell{c}{1.1075e-10 \\ 1.1084e-10} & \tabincell{c}{245 \\ {\bf 19}} & \tabincell{c}{1.73\\ 0.17} \\
						\hline
						\bf ${10^{-5}}/{m}$    &\tabincell{c}{Lawson's iteration \\Interior-point method}&\tabincell{c}{1000\\55} & \tabincell{c}{1.0529e-5 \\1.0528e-5} &\tabincell{c}{1.1075e-10 \\1.1084e-10} & \tabincell{c}{229 \\ {\bf 19}} & \tabincell{c}{1.48\\ 0.22} \\
						\hline
						\bf ${10^{-4}}/{m}$    &\tabincell{c}{Lawson's iteration \\Interior-point method}&\tabincell{c}{1000\\67} & \tabincell{c}{1.0529e-5 \\1.0528e-5} &\tabincell{c}{1.1075e-10\\1.1084e-10} & \tabincell{c}{221\\ {\bf 19}} & \tabincell{c}{1.55\\0.14} \\
						\hline 
						\hline
						&\multicolumn{6}{|c|}{ \tabincell{c}{Test on $g_2$ \eqref{eq:eg2Cfun} on $\bbP_{21}$ and $m=2001$}}\\
						\hline
						$\epsilon_w$&   {\rm  Methods} &  {\rm \# of iterations $k$} & {\tt $\Vert \br^{(k)}\Vert_{\infty}$} & {\tt $d(\bw^{(k)})$} & {\rm \# of final nodes} & {\rm CPU (s) } \\
						\hline
                        \bf ${10^{-6}}/{m}$    &\tabincell{c}{Lawson's iteration \\Interior-point method} &\tabincell{c}{1000\\28} & \tabincell{c}{1.8297e-2\\ 1.8294e-2} &\tabincell{c}{3.3436e-4\\3.3469e-4} & \tabincell{c}{597 \\ {\bf 31}} & \tabincell{c}{3.55\\ 0.09} \\
                        \hline
                        \bf ${10^{-5}}/{m}$    &\tabincell{c}{Lawson's iteration \\Interior-point method}&\tabincell{c}{1000\\28} & \tabincell{c}{1.8297e-2\\ 1.8294e-2} &\tabincell{c}{3.3436e-4\\3.3469e-4} & \tabincell{c}{561 \\ {\bf 31}} & \tabincell{c}{3.59\\ 0.09} \\
                        \hline
                        \bf ${10^{-4}}/{m}$    &\tabincell{c}{Lawson's iteration \\Interior-point method}&\tabincell{c}{1000\\27} & \tabincell{c}{1.8297e-2\\ 1.8294e-2} &\tabincell{c}{3.3436e-4\\3.3469e-4} & \tabincell{c}{529\\ {\bf 31}} & \tabincell{c}{3.72\\ 0.08} \\
                        \hline 
						&\multicolumn{6}{|c|}{ \tabincell{c}{Test on $g_2$ \eqref{eq:eg2Cfun}  on $\bbP_{31}$ and $m=2001$}}\\
						\hline
						$\epsilon_w$&   {\rm  Methods} &  {\rm \# of iterations $k$} & {\tt $\Vert \br^{(k)}\Vert_{\infty}$} & {\tt $d(\bw^{(k)})$} & {\rm \# of final nodes} & {\rm CPU (s) } \\
						\hline
                         \bf ${10^{-6}}/{m}$    &\tabincell{c}{Lawson's iteration \\Interior-point method} &\tabincell{c}{1000\\28} & \tabincell{c}{1.2449e-2\\ 1.2447e-2} &\tabincell{c}{1.5477e-4\\1.5493e-4}  & \tabincell{c}{540 \\ {\bf 32}} & \tabincell{c}{8.36\\ 0.30} \\
                         \hline
                         \bf ${10^{-5}}/{m}$    &\tabincell{c}{Lawson's iteration \\Interior-point method}&\tabincell{c}{1000\\28} & \tabincell{c}{1.2449e-2\\ 1.2447e-2} &\tabincell{c}{1.5477e-4\\1.5493e-4} & \tabincell{c}{512 \\ {\bf 32}} & \tabincell{c}{8.89\\ 0.33} \\
                         \hline
                         \bf ${10^{-4}}/{m}$    &\tabincell{c}{Lawson's iteration \\Interior-point method}&\tabincell{c}{1000\\27} & \tabincell{c}{1.2449e-2\\ 1.2447e-2} &\tabincell{c}{1.5477e-4\\1.5493e-4} & \tabincell{c}{474\\ {\bf 32}} & \tabincell{c}{6.75\\ 0.22} \\
                         \hline
					\end{tabular}
				\end{scriptsize}
		}\end{center}
	\end{table}

	\begin{figure}[thb]		
	\centering
	\includegraphics[width=0.9\linewidth,height=0.67\textheight]{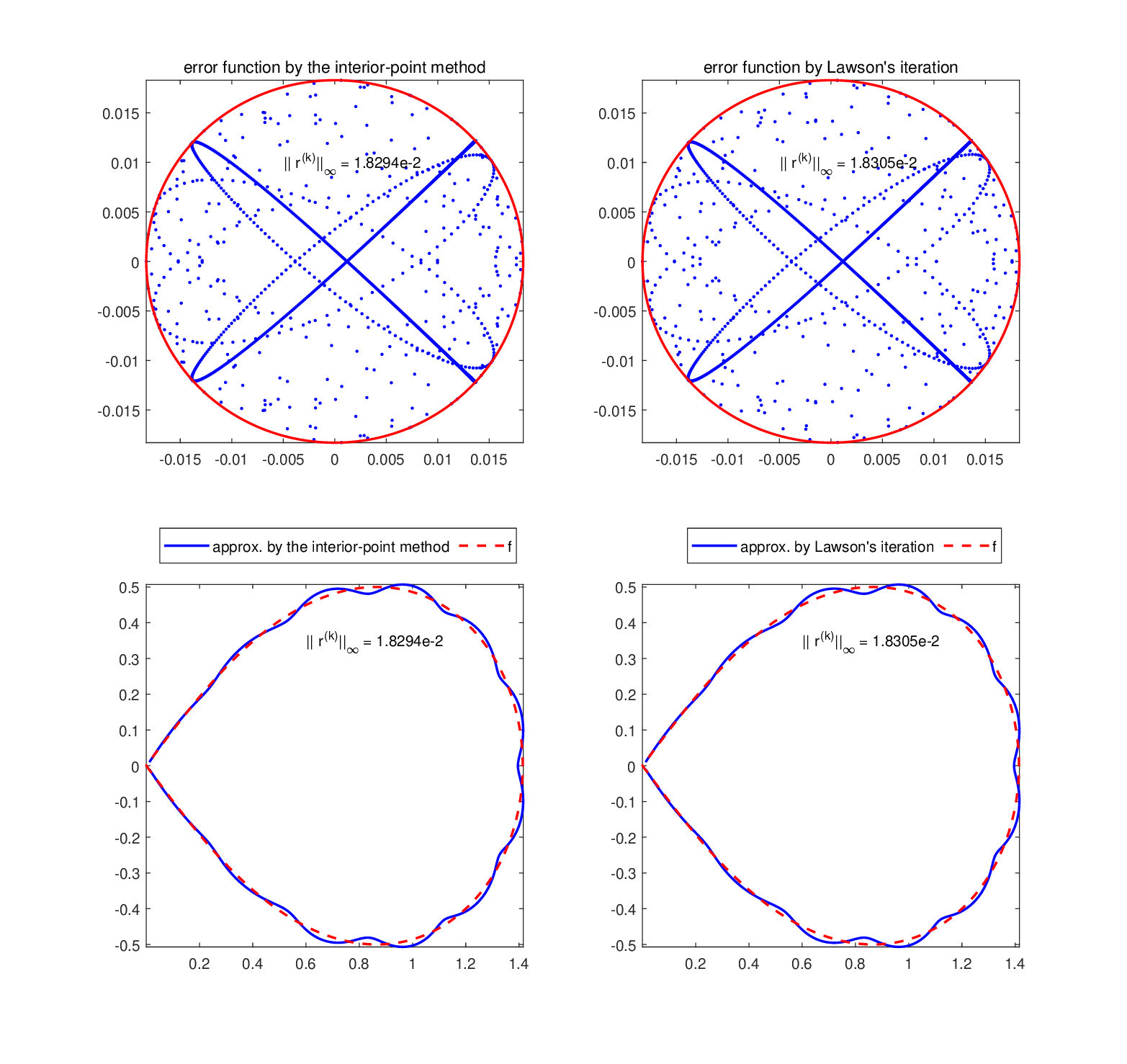}
	\vspace{-1cm}
	\caption{\small The error curve (top) and the approximation (bottom) by the interior-point method (left) and Lawson's iteration (right) for Example \ref{eg:5} on $\bbP_{21}$.}\label{fig:Example5}
    \end{figure}

	\begin{figure}[tbh!!!]
		\hskip-12mm 
		\includegraphics[width=1.17\linewidth,height=0.59\textheight]{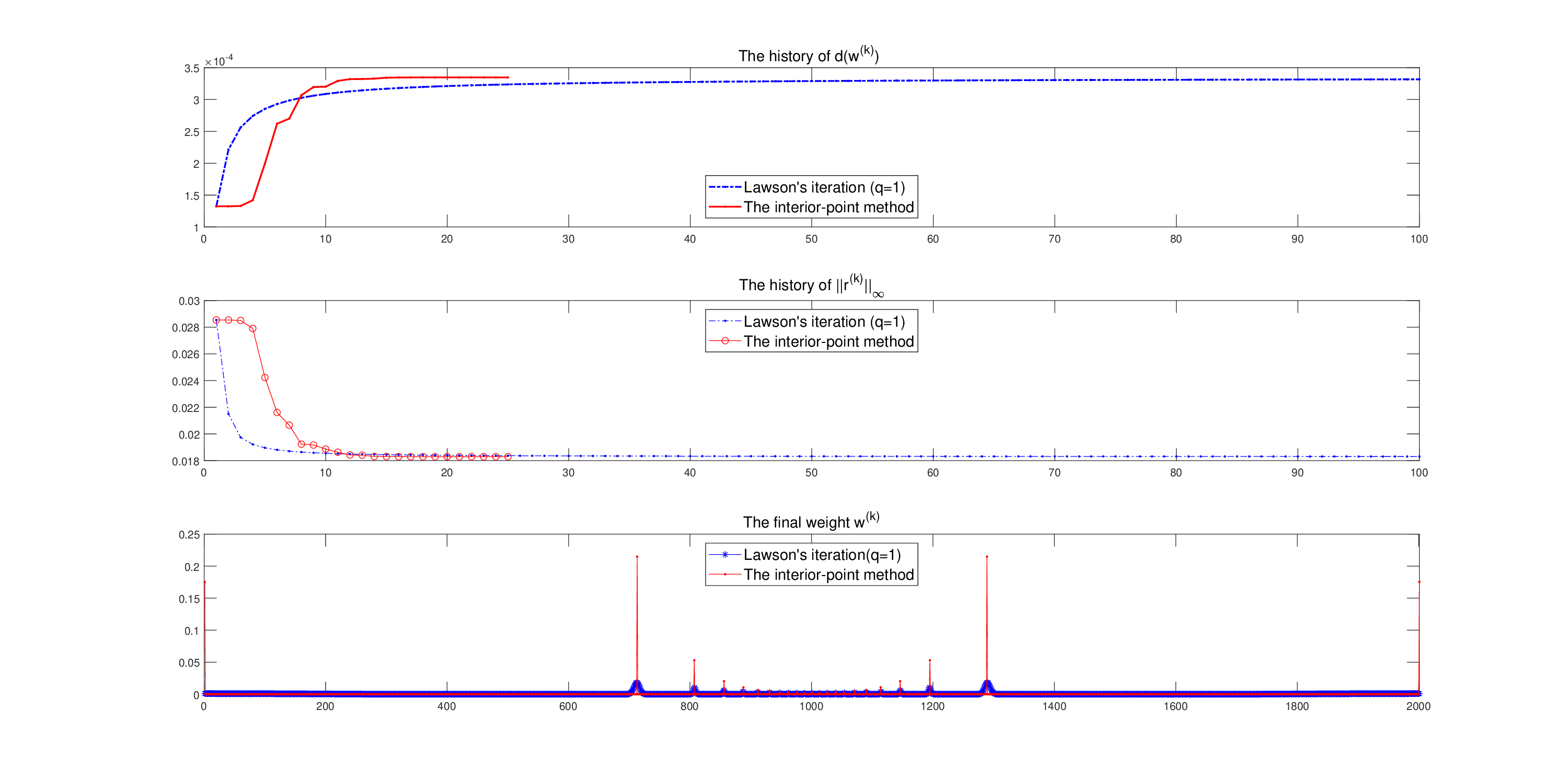}
		\vspace{-1cm}
		\caption{\small From top to bottom:$\{d(\bw^{(k)})\},\{\Vert \br^{(k)}\Vert_{\infty}\}$ and stopping weights $\{\bw_i\}_{i=1}^{2001}$ from Lawson's iteration(blue) and interior-point method(red) for Example \ref{eg:5} on $\bbP_{21}$.}
		\label{fig:Example5a}
	\end{figure}
\end{example}

\section{Concluding remarks}\label{sec:conclusion}
In this paper, from modern optimization theory, we reconsidered the classical linear Chebyshev approximation on a given finite set and  Lawson's iteration.  By an $L_q$-equivalent formulation \eqref{eq:q-equv} for $q\ge 1$, we developed the associated Lagrange $L_q$-weighted  dual problem. By investigating two particular $L_q$-norm cases (i.e., $L_1$ and $L_2$), we first provided  an elementary proof for the well-known Alternation Theorem  for the real case  based primarily upon the duality theory of  linear programming. In order to analyze, and further, to improve the traditional Lawson's iteration, we focus on the $L_2$-weighted dual programming. The monotonic convergence of the dual objective function for two commonly used weight-updates are revealed; more importantly, by relying on the equivalent $L_2$-weighted dual problem, Newton's type iteration, in the framework of the interior-point method,  has been developed. Fast convergence of the weights are observed and the method has the capability to find   the reference points (the equioscillation points in the Alternation Theorem for the real case) that characterize the unique minimax approximation.

\section*{Data availibility} No datasets are used in the paper.

\section*{Declarations}

{\bf Conflict of interest} The authors declare that they have no conflict of interest.
 
 {\small 
 \def\noopsort#1{}\def\l{\char32l}\def\v#1{{\accent20 #1}}
  \let\^^_=\v\def\hbk{hardback}\def\pbk{paperback}
\providecommand{\href}[2]{#2}
\providecommand{\arxiv}[1]{\href{http://arxiv.org/abs/#1}{arXiv:#1}}
\providecommand{\url}[1]{\texttt{#1}}
\providecommand{\urlprefix}{URL }

 }


%
\end{document}